\documentclass{article}
\usepackage{amsmath,amssymb,amsthm,graphicx,epsfig,float,url,subfigure}
\usepackage{mathrsfs}
\usepackage[colorlinks=true]{hyperref}
\usepackage{pdfsync}
\usepackage{color}
\usepackage{xspace}
\usepackage{enumitem}
\usepackage{verbatim}
\usepackage[applemac]{inputenc}
 \usepackage{leftidx}

\graphicspath{{fig/}}

\def\R{\mathbb{R}}
\def\N{\mathbb{N}}
\def\Z{\mathbb{Z}}
\def\Q{\mathbb{Q}}

\renewcommand{\geq}{\geqslant}
\renewcommand{\leq}{\leqslant}

\newtheorem{theorem}{Theorem}[section]
\newtheorem{proposition}[theorem]{Proposition}
\newtheorem{corollary}[theorem]{Corollary}
\newtheorem{lemma}[theorem]{Lemma}

\newtheorem{thmbis}{Theorem}
\newtheorem{thmter}{Theorem}

\newcounter{thmbiss}
\newcounter{lemmabiss}

\theoremstyle{definition}
\newtheorem{definition}[theorem]{Definition}

\newtheorem{remark}[theorem]{Remark}

\newcommand{\Norm}[2]{\| #1 \|_{#2}}

\renewcommand{\d}{\partial}
\DeclareMathOperator{\Char}{Char}

\newcommand{\rhs}{r.h.s.\@\xspace}

\newcommand{\ie}{i.e.\@\xspace}
\newcommand{\eg}{e.g.\@\xspace}

\newcommand{\resp}{resp.\@\xspace}
\newcommand{\suff}{sufficiently\xspace}
\newcommand{\nhd}{neighborhood\xspace}

\newcommand{\hf}{\frac12}

\newcommand{\ovl}[1]{\overline{#1}}
\newcommand{\udl}[1]{\underline{#1}}

\newcommand{\ic}{_{|t=0}}
\newcommand{\bld}[1]{\mbox{\boldmath $#1$}}

\newcommand{\Con}{\ensuremath{\mathscr C}}

\newcommand{\Cinfc}{\ensuremath{\Con^\infty}_{c}}

\newcommand{\eps}{\varepsilon}

\newcommand{\comp}{\leftidx{^b}{}{}}
\title{Geometric control condition for the wave equation with a
  time-dependent observation domain\footnote{G.~Lebeau
  acknowledges the  support of the European Research Council, 
ERC-2012-ADG, project number 320845:  Semi Classical Analysis of Partial Differential
Equations. E.~Tr\'elat
  acknowledges the support by FA9550-14-1-0214 of the EOARD-AFOSR.}}

\author{
J\'er\^ome  Le Rousseau\footnote{Univ. Paris-Nord, CNRS UMR
  7339,  Laboratoire analyse, g\'eom\'etrie et applications (Laga), 
 Institut
  universitaire de France,
99 avenue Jean Baptiste cl\'ement,
93430 Villetaneuse, France
  (\texttt{jlr@math.univ-paris13.fr}).},
\and
Gilles Lebeau\footnote{Univ. de Nice Sophia-Antipolis,  CNRS UMR
  7351, Laboratoire J.-A. Dieudonn\'e, Parc Valrose, 06108 Nice Cedex
  02, France (\texttt{gilles.lebeau@unice.fr}).},
\and
Peppino Terpolilli\footnote{ \mbox{Total, Centre scientifique et technique Jean-Feger, av. Larribau, 64000 Pau, France} (\texttt{peppino.terpolilli@total.com}).},
\and
and Emmanuel Tr\'elat\footnote{Sorbonne Universit\'es, UPMC Univ Paris
  06, CNRS UMR 7598, Laboratoire Jacques-Louis Lions, Institut
  universitaire de France, F-75005, Paris, France
  (\texttt{emmanuel.trelat@upmc.fr}).}
}

\begin{document}

\maketitle

\begin{abstract}
  We characterize the observability property (and, by duality, the
  controllability and the stabilization) of the wave equation on a
  Riemannian manifold $\Omega,$ with or without boundary, where the
  observation (or control) domain is time-varying. We provide a
  condition ensuring observability, in terms of propagating
  bicharacteristics. This condition extends the well-known geometric
  control condition established for fixed observation domains.

  As one of the consequences, we prove that it is always possible to
  find a time-dependent observation domain of arbitrarily small
  measure for which the observability property holds.  From a
  practical point of view, this means that it is possible to
  reconstruct the solutions of the wave equation with only few sensors
  (in the Lebesgue measure sense), at the price of moving the sensors
  in the domain in an adequate way.

We provide several illustrating examples, in which the observation
domain is the rigid displacement in $\Omega$ of a fixed domain, with
speed $v,$ showing that the observability property depends both on $v$
and on the wave speed. Despite the apparent simplicity of some of our
examples, the observability property can depend on nontrivial arithmetic
considerations.

\medskip
\noindent\textbf{Keywords:} wave equation, geometric control condition, time-dependent observation domain.

\noindent\textbf{AMS classification:} 
35L05,  %Wave equation
93B07,  %Observability
93C20.  %Systems governed by partial differential equations
\end{abstract}

\small
\setcounter{tocdepth}{2}
\tableofcontents

\normalsize
\section{Introduction and main result}%\label{secintro}

\subsection{Framework}
\label{sec: framework}
Studies of the stabilization and the controllability for the wave
equation go back to the 70's with the works of D. L. Russell (see, \eg,
\cite{Russell:71,Russell:71b}). The work of J.-L.~Lions was very
important in the formalization of many controllability questions (see,
\eg, \cite{lions}).
In the case of a  manifold
without boundary $\Omega$, 
the pioneering work of J.~Rauch and M.~Taylor~\cite{RT:74} related the question of
fast stabilization, that is, exhibiting an exponential decay of the energy, to a geometric condition connecting the damping region
$\omega \subset \Omega$
and the rays of geometrical optics, 
the now celebrated Geometric Control Condition (in short, GCC).
The damped wave equation takes the form
\begin{align*}
  \d_t^2 u - \triangle u + \chi_\omega \d_t u =0.
\end{align*}
 Using that the
energy of the solution to a hyperbolic equation is
largely carried along the rays, if one assumes that any
ray will have reached the region $\omega$ where the operator is
dissipative in a finite time, one can prove that the energy decays
exponentially in time, with an additional unique continuation argument that
allows one to handle the low frequency part of the energy. The work \cite{RT:74}
only treated the case of a manifold
$\Omega$ 
without boundary, leaving open the case of
manifolds with boundary until the work of C.~Bardos,
G.~Lebeau, and J.~Rauch \cite{BardosLebeauRauch}.  The
understanding of the propagation of singularities in the presence of
the  boundary $\d\Omega$, after the seminal work of R.~Melrose and
J.~Sj\"ostrand~\cite{MelroseSjostrand1,MelroseSjostrand2}, was a key element in the proof of \cite{BardosLebeauRauch}, providing
a generalized notion of rays, taking reflections at the boundary into
account as well as glancing and gliding phenomena. 
The geometric condition for $\omega$, now an open subset of
$\ovl{\Omega}$,  is then the requirement that every generalized ray
should meet the damping region $\omega$ in a finite time. The resulting stabilization estimate then takes the form:
\begin{align*}
  E_0(u(t))   \leq C e^{-C' t}   E_0(u(0)), 
\end{align*}
where $E_0$ is the following energy
\begin{align*}
  E_0(u(t))  = \Norm{u(t)}{H^1(\Omega)}^2 + \Norm{\d_t u(t)}{L^2(\Omega)}^2.
\end{align*}
Note that, if an open set $\omega$ does not fulfill
the geometric control condition, then only a logarithmic type of energy
decay can be achieved in general \cite{Lebeau:96,LR:97,Burq:98}.\footnote{In fact, intermediate decay rates have been established in particular geometrical settings, see for instance \cite{Schenck} for almost exponential decay, \cite{BurqHitrik,Phung,AL} for polynomial decay.}

The question of exact controllability relies on the same line of
arguments as for the exponential stabilization.
By exact controllability in time $T>0$, for the control wave equation
\begin{align*}
  \d_t^2 u - \triangle u = \chi_\omega(x) f,
\end{align*}
one means, given an arbitrary
initial state $(u_0, u_1)$ and an arbitrary final state
$(u^F_0, u^F_1)$, the ability to find $f$ such that 
$ (u_{t=T}, \d_t u_{t=T}) = (u^F_0, u^F_1)$ starting from
$ (u_{t=0}, \d_t u_{t=0}) = (u_0, u_1)$. If the energy level is
$(u(t), \d_t u(t)) \in H^1(\Omega) \oplus L^2(\Omega)$, it is natural
to seek $f \in L^2((0,T)\times \Omega)$. Boundary conditions can be of
Dirichlet or Neumann types. 

\medskip
In fact, as is well known, both exponential stabilization
and exact controllability of the wave equation in a domain
$\Omega$, with a damping or a control only acting in an open  region
$\omega$ of $\ovl{\Omega}$, are equivalent to
an observability estimate for a free wave. For such a wave, the energy
is constant with respect to time. The observability inequality takes
the following form:
for some constant $C>0$ and some $T>0$, we have
\begin{align}
  \label{eq: intro observability}
  E_0(u) 
  \leq C \int_0^T \Norm{\d_t u}{L^2(\omega)}^2 dt.
\end{align}
For the issue of exact controllability, the time $T>0$ in this
inequality is then the control time (horizon). 
If the open set $\omega$ fulfills the Geometric Control Condition, then the results of
\cite{RT:74,BardosLebeauRauch} show that the infimum of all possible
such times $T$ coincides with the infimum of all possible times in the
Geometric Control Condition. 
Note however that there are cases in which the Geometric Control Condition does not hold,
and yet the observability inequality \eqref{eq: intro observability} is valid: the case $\Omega$ is a sphere and $\omega$ is a half-sphere is a typical example.

\medskip A glance at  inequality~\eqref{eq: intro observability},
shows that observability is in fact to be understood as occuring in a
space-time domain, here $(0,T) \times \omega$.  It is then natural to
wonder if observability can hold if it is replaced by some other open
subset of $(0,T)\times \ovl{\Omega}$.  This is the subject of the
present article.  

\medskip The motivation for such a study can be seen as fairly
theoretical. However, in practical issues, in different industrial
contexts, for nondestructive testing, safety applications, as well as
tomography techniques used for imaging bodies (human or not), this
question becomes quite relevant.  In fact, the industrial framework of
seismic exploration was the original motivation for this work. In the
different fields we mentionned, data are collected to be exploited in an
interpretation step which involves the solution of some inverse
problem. The point is that the device used to collect data does
not fit well with the usual geometric condition which is
crucial to obtain an observability result.  In some cases it appears of great
interest to be able to tackle situations where the observation set is
time-dependent. In others, the reduction of data volume may be sought,
while preserving the data quality.
One may also face a situation in which all sensors cannot be 
active at the same time. 

The example of seismic data acquisition can help the
reader get a grasp on the industrial need to better design data
acquisition procedures.  In the case of a towed marine seismic data
acquisition campaign, a typical setup consists in six parallel
streamers with length $6000$~m, separated by a distance of $100$~m,
floating at a depth of $8$~m. The basic receiving elements are pressure
sensitive hydrophones composed of piezoelectric ceramic crystal devices
that are placed some $20$ to $50$~m apart along each streamer. A source (a
carefully designed air gun array) is shot every $25$~m while the boat
moves. The seismic data experiment lasts around $8$~s.
One
understands with this description that a huge amount (terabytes) of
data is recorded during one such acquisition campaign above an area of
interest beneath the sea floor.  Of course, the velocity of the ship and of the streamers
is very small as compared to that of the seismic waves ($1500$~m/s in
water and up to $5000$~m/s for examples in salt bodies that are typical
in the North Sea or in the Gulf of Mexico). Yet, however small it may
be, one can question its impact on the quality of the data. 
One could also want
not to use all receivers at a single time but rather to design a
dynamic (software based) array of receivers during the time of the
seismic experiment. The reader will of course
realize that the mathematical results we present here are very far
from solving this problem. They however give some leads on what 
important theoretical issues can be.

\medskip

An inspection of the proof of \cite{BardosLebeauRauch} shows that it uses the
invariance of the observation cylinder $(0,T)\times \omega$ with respect to time in a crucial
way. Hence, the method, if not modified, cannot be applied to a general open subset of
$(0,T)\times \ovl{\Omega}$. One of the contributions of
the present work is to remedy this issue. In fact, this is done by a significant
simplification of the argument of \cite{BardosLebeauRauch}, yielding
a less technical aspect in one of the steps of the proof.
Eventually, the result that we obtain is in fact faithful to the intuition one may
have. The proper geometric condition to impose on an open subset 
$Q$ of $(0,T)\times \ovl{\Omega}$, for an observability condition of
the form
\begin{align}
  \label{eq: intro observability 2}
  E_0(u) 
  \leq C \iint_Q |\d_t u|^2\, dt\, d x
\end{align}
to hold is the following: for any generalized ray $t \mapsto x(t)$ initiated at time
$t=0$, there should be a time $0 <  t_1< T$ such that the ray is
located in $Q \cap \{ t=t_1\}$, that is $(t_1, x(t_1)) \in Q$.  
This naturally generalizes the usual Geometric Control Condition in
the case where $Q$ is the cylinder $(0,T)\times \omega$. 

One of the interesting consequences of our analysis lies in the
following fact:
if the geometric condition holds for a time-dependent domain
$Q$, a thinner domain for which the condition holds as well can be simply
obtained by picking a \nhd of the boundary of $Q$.  This can be viewed
as a step towards the reduction of the amount of data collected in the
practical applications mentioned above.

We complete our analysis with a set of examples in {\em very simple}
geometrical situations. Some of these examples show that even if ``many''
rays are missed by a static domain, a moving version of this domain
can capture in finite time all rays, even with a very slow
motion. However other examples show situations in which ``very few'' rays
are missed, and a slow motion of the observation set allows one to
capture these rays, yet implying that other rays remain away from the
moving observation region for any positive time. Those examples may
become hard to analyse because of the complexity of the Hamiltonian
dynamics that governs the rays. Yet, they illustrate that naive
strategies can fail to achieve the fulfillment of the Geometric
Control Condition. Those examples show that further study would be of
interest, with a study of the increase or decrease of the minimal
control time as an observation set is moved around.
Some examples show that this minimal control time may not be continuous
with respect to the dynamics we impose on a moving control region.

\subsection{Setting}\label{sec_1.2}
Let  $(M,g)$ be a smooth $d$-dimensional Riemannian manifold, with $d \geq 1$.
Let %$T$ be a positive real number and let 
$\Omega$ be an open bounded connected subset of $M$, with a smooth boundary if $\partial\Omega\neq\emptyset$. 
We consider the wave equation
\begin{equation}\label{waveEq}
\partial_t^2u-\triangle_g u=0,
\end{equation}
in $\R\times\Omega$. %$(0,T)\times\Omega$. 
Here, $\triangle_g$ denotes the Laplace-Beltrami operator on $M$, associated with the metric $g$ on $M$. 
If the boundary $\partial\Omega$ of $\Omega$ is nonempty, then we
consider boundary conditions of the form
\begin{equation}\label{condBC}
Bu=0\quad\textrm{on}\ \R\times\partial\Omega, %(0,T)\times\partial\Omega,
\end{equation}
where the operator $B$ is either:
\begin{itemize}
\item the Dirichlet trace operator, $Bu=u_{\vert\partial\Omega}$;
\item or the Neumann trace operator, $Bu={\d_n u}_{|\partial\Omega}$, where $\d_n$ is the outward normal derivative along $\partial\Omega$.
\end{itemize}
Our study encompasses the case where $\partial\Omega=\emptyset$: in this case, $\Omega$ is a compact connected $d$-dimensional Riemannian manifold.
%The canonical Riemannian volume on $M$ is denoted by $V_g$, inducing the canonical measure $dx_g$. 
Measurable sets are considered with respect to the Riemannian measure $dx_g$ (if $M$ is the usual Euclidean space $\R^n$ then $dx_g$ is the usual Lebesgue measure).

In the case of a manifold without boundary or in the case of homogeneous Neumann
 boundary conditions, the Laplace-Beltrami operator is not invertible on $L^2(\Omega)$ but is invertible in 
$$
L^2_{0}(\Omega)=\Big\{u\in L^2 (\Omega) \ \vert \ \int_{\Omega}u(x)\, dx_g=0\Big\}.
$$
In what follows, we set $X=L^2_{0}(\Omega)$ in the boundaryless or in
the Neumann case, and $X=L^2 (\Omega)$ in the Dirichlet case (in both
cases, the norm on $X$ is the usual $L^2$-norm). We denote
by $A=-\triangle_g$ the Laplace operator defined on $X$ with domain
$D(A)=\{u\in X\ \vert\ Au\in X\ \textrm{and}\ Bu=0 \}$ with one of the
above boundary conditions whenever $\partial\Omega\neq\emptyset$.

Note that $A$ is a selfadjoint positive operator.  In the case of
Dirichlet boundary conditions, $X= L^2(\Omega)$ and we have
$D(A)=H^2(\Omega)\cap H^1_0(\Omega)$, $D(A^{1/2})=H^1_0(\Omega)$ and
$D(A^{1/2})'=H^{-1}(\Omega)$, where the dual is considered with
respect to the pivot space $X$. For Neumann boundary conditions,
$X = L^2_0(X)$ and we have
$D(A)=\{u\in H^2(\Omega) \cap  L^2_0(X) \ \vert \ \frac{\partial u}{\partial
  n}_{\vert\partial\Omega}=0\}$
and
$D(A^{1/2})= H^1(\Omega) \cap  L^2_0(X)$.
The Hilbert spaces $D(A)$, $D(A^{1/2})$, and $D(A^{1/2})'$ are
respectively endowed with the norms
$\Norm{u}{D(A)} = \Norm{Au}{L^2(\Omega)}$,
$\Norm{u}{D(A^{1/2})} = \Norm{A^{1/2}u}{L^2(\Omega)}$ and
$\Norm{u}{D(A^{1/2})'} = \Norm{A^{-1/2}u}{L^2(\Omega)}$.

\medskip

For all $(u^0,u^1)\in D(A^{1/2}) \times X$ (\resp
$X \times D(A^{1/2})'$), there exists a unique solution
$u \in \Con^0(\R;D(A^{1/2})) \cap \Con^1(\R;X)$ (\resp,
$u \in \Con^0(\R;X) \cap \Con^1(\R;D(A^{1/2})')$) of
\eqref{waveEq}-\eqref{condBC} such that $u\ic=u^0$ and
$\partial_t u\ic=u^1$.  In both cases, such solutions of
\eqref{waveEq}-\eqref{condBC} are to be understood in a weak sense.

\begin{remark}
  In \eqref{waveEq}, we consider the classical d'Alembert wave operator
  $\square_g = \partial_t^2-\triangle_g$. In fact, the results of the present article remain
  valid for the more general wave operators of the form 
  $$ P = \partial_t^2 - \mathop{\textstyle{\sum}}_{i,j}
  a_{ij}(x) \partial_{x_i} \partial_{x_j} + \textrm{lower-order terms}
  ,$$
  where  $(a_{ij}(x))$ is a smooth  real-valued symmetric positive definite matrix,
  and where the lower-order terms are smooth and do not depend on $t$.
  We insist on the fact that our
  approach is limited to operators with time-independent coefficients
  as in \cite{BardosLebeauRauch}.
\end{remark}

\subsection{Observability}
Let $Q$ be an open subset of $\R\times \ovl{\Omega}$. We denote by $\chi_{Q}$
the characteristic function of $Q$, defined by $\chi_{Q}(t,x)=1$ if
$(t,x)\in Q$ and $\chi_{Q}(t,x)=0$ otherwise. We set
$$
\omega(t) = \{ x\in\Omega\ \mid\  (t,x)\in Q\},
$$
so that
$
Q = \{ (t,x)\in\R\times\Omega\ \mid\  t\in\R,\ x\in\omega(t)\} .
$
Let $T>0$ be arbitrary. We say that \eqref{waveEq}-\eqref{condBC} is observable on $Q$ in time $T$ if there exists $C>0$ such that
\begin{equation}\label{observability}
C \Norm{ (u^0,u^1)}{ D(A^{1/2}) \times X}^2 
\leq
\Norm{\chi_{Q} \d_t u}{L^2((0,T)\times\Omega)}^2 
= \int_0^T\!\! \int_{\omega(t)} \vert \d_t u(t,x)\vert^2 \, dx_g\, dt,
\end{equation}
for all $(u^0,u^1)\in D(A^{1/2})\times X$, where $u$ is the
solution of \eqref{waveEq}-\eqref{condBC} with initial conditions
$u\ic = u^0$ and $\d_t u \ic = u^1$. One refers to \eqref{observability} as to an
\textit{observability inequality}.

\medskip
The observability inequality~\eqref{observability} is stated for
initial conditions $(u^0,u^1)\in D(A^{1/2})\times X$. Other energy
spaces can be used. An important example is the following proposition.
\begin{proposition}
  \label{prop: equivalence observability}
  The observability inequality~\eqref{observability} is equivalent to 
  having $C>0$ such that 
  \begin{equation}\label{observability-shift}
    C \Norm{ (u^0,u^1)}{ X \times D(A^{1/2})'}^2 
    \leq
    \Norm{\chi_{Q} u}{L^2((0,T)\times\Omega)}^2 
    = \int_0^T\!\! \int_{\omega(t)} \vert u(t,x)\vert^2 \, dx_g\, dt,
  \end{equation}
  for all $(u^0,u^1)\in X \times D(A^{1/2})'$, where $u$ is the
solution of \eqref{waveEq}-\eqref{condBC} with initial conditions
$u\ic = u^0$ and $\d_t u \ic = u^1$.
\end{proposition}
This proposition is proven in Section~\ref{sec_rem_obs_equiv}. 

\medskip

In the existing literature, the observation is most often made on
cylindrical domains $Q=(0,T) \times \omega$ for some given $T>0$,
meaning that $\omega(t) = \omega$. In such a case where the observation
domain $\omega$ is stationary, it is known that, within the class of
smooth domains $\Omega$, the observability property holds if the pair
$(\omega,T)$ satisfies the \textit{Geometric Control Condition} (in
short, GCC) in $\Omega$ (see \cite{BardosLebeauRauch,BurqGerard}).
Roughly speaking, it says that every geodesic propagating in $\Omega$
at unit speed, and reflecting at the boundary according to the
classical laws of geometrical optics, so-called generalized geodesics,
should meet the open set $\omega$ within time $T$.

 In the present article, our goal is to extend the GCC to \textit{time-dependent observation domains}.
For a precise statement of the GCC, we first recall the definition of
generalized geodesics and bicharacteristics.

\subsubsection{The generalized bicharacteristic flow of R.~Melrose and
  J.~Sj\"ostrand}
\label{sec: generalized bichar}

First, we define generalized bicharacteristics in the interior of
$\Omega$. There, they coincide with the classical notion of
bicharacteristics. Second, we define generalized
bicharacteristics in the \nhd of the boundary.

\paragraph{Symbols and bicharacteristics in the interior.}

The principal symbol of $-\triangle_g$ coincides with the cometric $g^*$ defined by 
$$
g^*_x(\xi,\xi)=\max_{v\in T_xM\setminus\{0\}} \frac{\langle\xi,v\rangle^2}{g_x(v,v)} ,
$$ 
for every $x\in M$ and every $\xi\in T_x^*M$. In local coordinates, we
denote by $g_{ij}(x)$ the Riemannian metric $g$ at point $x$, that is,
$
g(v,\tilde{v})(x) = g_{ij}(x) v^i(x) \tilde{v}^j(x),
$
for $v,\tilde{v} \in T M$, that is,  two vector fields, 
 and by
$g^{ij}(x)$ the cometric $g^*$ at $x$, that is 
$
g^*(\omega,\tilde{\omega})(x) = g^{ij}(x) \omega_i(x) \tilde{\omega}_j(x),
$
for $\omega, \tilde{\omega} \in T^* M$, that is, two 1-forms.
In local coordinates, the Laplace-Beltrami reads
\begin{equation*}
  -\triangle_g =  -{\mathsf g}(x)^{-1/2}\d_i ( {\mathsf g}(x)^{1/2} g^{ij}(x) \d_j ).
\end{equation*}

In $\R \times M$, the principal symbol of the wave operator
$\d_t^2 -\triangle_g$ is then $p (t,x, \tau,\xi) = -\tau^2 +
g^*_x(\xi,\xi)$.
In $T^\ast(\R\times M)$, the Hamiltonian vector field $H_p$ associated with $p$ is given by  $H_p f = \{
p, f\}$, for $f \in \Con^1(T^\ast(\R\times M))$. In local
coordinates, $H_p$ reads
\begin{align*}
  H_p = \d_\tau p\, \d_t + \nabla_{\xi}p\,  \nabla_{x}  - \nabla_{x}p\,
  \nabla_{x}
  =- 2 \tau \d_t  + 2 g^{jk}(x)\xi_k \d_{x_j} 
  -\d_{x_j}g^{ik}(x) \xi_i\xi_k \d_{\xi_j},
\end{align*}
with the usual Einstein summation convention.
Along the  integral curves of $H_p$, the value of $p$ is constant as $H_p p
=0$.  Thus, the characteristic set $\Char(p) = \{ p =0\}$ is invariant under the flow of $H_p$. 
In $T^\ast(\R\times M)$, bicharacteristics are
defined as the maximal  integral curves of $H_p$ that lay in $\Char(p)$. 
The projections of the bicharacteristics onto $M$, using the variable
$t$ as a parameter, coincide with the geodesics on $M$ associated with the
metric $g$ travelled at speed one.

\bigskip
We set  $Y = \R\times \ovl{\Omega}$. 
 We denote by $\Char_Y(p)$ the characteristic set of $p$
above $Y$, given by
\begin{align*}
  \Char_Y(p) = \{ \rho = (t,x, \tau,\xi) \in T^\ast(\R \times M)
  \setminus {0}\ \mid\  x \in \ovl{\Omega} \
  \text{and} \  p( \rho) =0\}. 
\end{align*}

\paragraph{Coordinates and Hamiltonian vector fields near and at the boundary.}
Close to the boundary $\R \times \d \Omega$, using normal geodesic
coordinates $(x',x_d)$, the principal symbol of the Laplace-Beltrami
operator reads $\xi_d^2 + \ell(x,\xi')$. Set $y = (t,x)$, $y'=
(t,x')$ and $y_n= x_d$.  Here $n =
d+1$. In these coordinates,  the principal symbol of the wave operator
takes the form $p(y',y_n,\eta',\eta_n) = \eta_n^2 + r (y, \eta')$,
where $r$ is a smooth $y_n$-family of tangential (differential)
symbols, and the boundary $\R \times \d \Omega$ is locally
parametrized by $y'$ and given by $\{ y_n=0\}$. The open set
$\R \times \Omega$ is locally given by $\{ y_n >0\}$. 

The variables $\eta= (\eta',\eta_n)$ are the cotangent variables
associated with $y=(y',y_n)$.  We set
\begin{equation*}
  \d T^\ast Y  = \{ \rho = (y,\eta) \in T^\ast(\R \times M)\ \mid\ y_n =0\},
\end{equation*} as the
boundary of $T^\ast Y  = \{ \rho = (y,\eta) \in T^\ast(\R \times M)\ \mid\
y \in Y \}$.  In those local coordinates, the associated
Hamiltonian vector field $H_p$ is given by
\begin{equation*} 
  H_p = \nabla_{\eta'} r \nabla_{y'} 
  + 2 \eta_n \d_{y_n} - \nabla_{y} r \nabla_{\eta}.
\end{equation*}
We denote by $r_0$ the trace of $r$ on $\d T^\ast Y$, that is, 
$r_0 (y', \eta') = r (y', y_n=0, \eta')$.  We then introduce the
Hamiltonian vector field above the submanifold $\{ y_n =0\}$
\begin{equation*}  
  H_{r_0} = \nabla_{\eta'} r_0 \nabla_{y'} 
  - \nabla_{y'} r_0 \nabla_{\eta'}.
\end{equation*}

\paragraph{The compressed cotangent bundle.}
On $Y$, for  points $y=(y',y_n)$ near the boundary,  we define the vector
fiber bundle  $\comp T Y = \cup_{y \in Y} \comp T_y Y$, generated by the vector fields 
$\d_{y'}$ and $y_n \d_{y_n}$, in the local coordinates introduced
above.  We then have the natural map 
\begin{align*}   
  j: T^\ast Y &\longrightarrow  \comp T^\ast  Y = \displaystyle\bigcup_{y \in Y}  (\comp T_y
  Y)^\ast, \\
  (y; \eta)  &\longmapsto  (y; \eta', y_n\eta_n),
\end{align*}
expressed here in local coordinates for simplicity.
In particular: 
\begin{itemize} 
\item If $y \in \R \times \Omega$ then $\comp T^\ast_y Y =  j(T^\ast_y Y)$ is isomorphic
  to $T^\ast_y Y = T^\ast_y (\R \times M)$;
\item if $y \in \R \times \d \Omega$ then $\comp T^\ast_y Y = j(T^\ast_y Y)$ is
  isomorphic to $T^\ast_y (\R \times \d \Omega)$.
\end{itemize} 
The bundle $\comp T^\ast Y$ is called the \emph{compressed cotangent bundle},
and we see that it allows one to patch together
$T^\ast_y (\R \times M)$ in the interior of $\Omega$ and
$T^\ast_y (\R \times \d \Omega)$ at the boundary in a smooth manner,
despite the discrepancy in their dimensions.

\paragraph{Decomposition of the characteristic set at the boundary.}
\label{sec: Gk}
We set $\Sigma  = j(\Char_Y(p)) \subset \comp T^\ast  Y$ and $\Sigma_0 = \Sigma_{|y_n
  =0} \subset \d\, \comp T^\ast  Y = \comp T^\ast  Y_{|y_n =0} \simeq T^\ast (\R \times \d \Omega)$.  Using local coordinates (for convenience here), we then
define $G \subset \Sigma_0$ by $r (y, \eta') = r_0 (y', \eta')=0$ as the
glancing set and $H = \Sigma_0 \setminus G$ as the hyperbolic set.
Hence, if $\rho =(y',y_n=0, \eta') \in \Sigma_0$ then 
\begin{align*}
  \rho \in H  \ \Leftrightarrow \ r_0 (y', \eta') < 0, \qquad 
  \rho \in G  \ \Leftrightarrow \ r_0 (y', \eta') = 0.
\end{align*}
The set of points $(y',y_n=0, \eta') \in \comp T^\ast  Y_{|y_n =0}$
such that $r_0 (y', \eta') > 0$ is referred to as the elliptic set $E$. 
We also set $\hat\Sigma = \Sigma \cup E = \Sigma \cup \comp T^\ast
Y_{|y_n=0}$ and the following cosphere quotient space $S^\ast \hat\Sigma = \hat\Sigma / (0,+\infty)$.

The glancing set is itself written as $G = G^2 \supset G^3\supset
\cdots \supset G^\infty$, with $\rho =(y',y_n=0, \eta)$ in $G^{k+2}$
if and only if 
\begin{align*}
  \eta_n = r_0 (\rho) =0, \quad H_{r_0}^j r_1 (\rho) =0, \ \ 0 \leq j < k,
\end{align*}
where $r_1 (\rho) = \d_{y_n} r (y', y_n=0, \eta')$.
Finally, we write $G^2 \setminus G^3$, the set of glancing points of
order exactly $2$, as the union of the diffractive set $G_d^2$ and of 
the gliding set $G_g^2$, that is $G^2 \setminus G^3= G_d^2 \cup G_g^2$, with 
\begin{align*}
  \rho \in G_d^2\  (\text{\resp},\ G_g^2)\ \Leftrightarrow\  \rho \in G^2
  \setminus G^3\ \text{and}\  r_1(\rho) < 0 \  (\text{\resp},\ > 0).
\end{align*}
Similarly, for $\ell \geq 2$, we write $G^{2\ell} \setminus G^{2\ell+1}$, the set of glancing points of
order exactly $k=2\ell$, as the union of the diffractive set $G_d^{2 \ell}$ and 
the gliding set $G_g^{2\ell}$, that is $G^{2\ell} \setminus
G^{2\ell+1}= G_d^{2\ell} \cup G_g^{2\ell}$, with 
\begin{align*}
  \rho \in G_d^{2\ell}\  (\text{\resp}\ G_g^{2\ell})\ \Leftrightarrow\
  \rho \in G^{2\ell}
  \setminus G^{2\ell+1}\ \text{and}\  H_{r_0}^{2\ell -2} r_1(\rho) < 0 \  (\text{\resp}\ > 0).
\end{align*}
We shall call diffractive\footnote{ In the sense of Taylor and Melrose the terminology diffractive only applies to
  $G_d^2$. Here, we chose to extend it to $\cup_{\ell \geq 1} G_d^{2
  \ell}$ as we shall refer to nondiffrative points, that is, the complement of
$G_d$, in
Section~\ref{sec: boundary control}. In the literature nondiffractive
points are defined this way but diffractive points are often defined
according to Taylor and Melrose. Then, the set of  nondiffractive
points and the set of diffractive points are not complements of one
another; a source of confusion.} a point in $G_d = \cup_{\ell \geq 1} G_d^{2 \ell}$.

Observe that a bicharacteristic going through a point of $G^k$ projects onto a
geodesic on $M$ that  has a contact of order $k$ with $\R \times \d \Omega$.

\paragraph{Generalized bicharacteristics.}
What we introduced above now allows us to give a precise definition of
generalized bicharacteristics above $Y$.

\begin{definition} 
A generalized bicharacteristic of $p$ is a differentiable
map
\begin{align*}
  \R \setminus B \ni s \mapsto \gamma(s) \in (\Char_Y(p) \setminus 
  \d T^\ast Y ) \cup
  j^{-1} (G), 
\end{align*}
where $B$ is a subset of $\R$ made of
isolated points, that  satisfies the following properties
\begin{enumerate}[label=(\roman*)]
  \item $\gamma'(s) = H_p(\gamma(s))$ if either $\gamma(s) \in \Char_Y(p)
    \setminus \d T^\ast Y$ or $\gamma(s) \in  j^{-1} (G_d)$; 
  \item $\gamma'(s) = H_{r_0}(\gamma(s))$ if   $\gamma(s) \in  j^{-1} (G \setminus G_d)$;
  \item If $s_0 \in B$, there exists $\delta>0$ such that $\gamma(s) \in
    \Char_Y(p) \setminus \d T^\ast Y$ for $s \in (s_0-\delta, s_0) \cup ( s_0,
    s_0+\delta)$. Moreover, the limits $\rho^\pm = (y^\pm, \eta^\pm) = \lim_{s \to s_0^\pm}\gamma(s)$
    exist  and  $y^-_n = y_n^+=0$, \ie, $\rho^\pm \in \Char_Y(p) \cap \d T^\ast Y$,
    $y^{-\prime} = y^{+\prime}$, $\eta^{-\prime} = \eta^{+\prime}$,
    and $\eta^{-}_n = - \eta^{+}_n$. That is, $\rho^+$ and $\rho^-$
    lay in the same hyperbolic fiber above a point in $\comp T^\ast
    Y_{|y_n=0}$: $j(\rho^+) = j(\rho^-) \in H$.
\end{enumerate}
\end{definition}
Point (i) describes the generalized bicharacteristic in the interior,
that is, in $T^\ast (\R \times \Omega)$, and at diffractive points,
where it coincides with part of a classical bicharacteristic as
defined above. Point (ii) describes the behavior in $G \setminus G_d$,
thus explaining that a generalized bicharacteristic can enter or leave
the boundary $\d T^\ast Y$ or locally remain in it. Point (iii)
describes reflections when the boundary $\d T^\ast Y$ is reached
transversally by a classical bicharacteristic, that is, at a point of
the hyperbolic set.  While $s\mapsto \xi(s)$ exhibits a jump at such a
point, $s\mapsto t(s)$ and $s\mapsto x(s)$ can both be extended by
continuity there. We shall thus proceed with this extension.
Above, for clarity we chose to state Point (iii) in local coordinates near the hyperbolic point.
The generalized bicharacteristics are however defined as a geometrical object, independent of the choice of coordinates.

\begin{definition}
Compressed generalized bicharacteristics are the image under the
map $j$ of the generalized bicharacteristics defined above.
\end{definition}

If $\comp \gamma = j(\gamma)$ is such a compressed generalized
bicharacteristics, then $\comp \gamma: \R \to \comp T^\ast
Y \setminus E$ is a continuous map 
(if one introduces
the proper natural topology on $\comp T^\ast Y$).

Using $t$ as parameter, generalized geodesics for $\Omega$, travelled
at speed one, are then the projection on $M$ of the
(compressed) generalized bicharacteristics. Generalized geodesics
remain in $\ovl{\Omega}$.  We shall call a ``ray'' this projection
following the terminology of geometrical optics.

\medskip
An important result is then the following.
\begin{proposition}
  \label{prop: uniqueness bicharacteristic flow}
  A (compressed) generalized bicharacteristic with no point in $G^\infty$ is
  uniquely determined by any one of its points. 
\end{proposition}
We refer to \cite{MelroseSjostrand1} for a proof of
this result and for many more details on generalized
bicharacteristics (see also \cite[Section 24.3]{Hoermander:V3}).

\subsubsection{A time-dependent geometric control condition}
\label{sec: t gcc}

With the notion of  compressed generalized bicharacteristic recalled in
Section~\ref{sec: generalized bichar}, 
we can state the geometric condition adapted to a time-dependent control domain.

\begin{definition}\label{def_tGCC}
  Let $Q$ be an open subset of $\R \times \ovl{\Omega}$, and
  let $T>0$.  We say that $(Q,T)$ satisfies the \emph{time-dependent
    Geometric Control Condition} (in short, $t$-GCC) if every
  generalized bicharacteristic
  $\comp \gamma: \R \ni s \mapsto (t(s), x(s),
  \tau(s), \xi(s)) \in \comp T^\ast Y
  \setminus E$
  is such that there exists $s \in \R$ such that
  $t(s) \in (0,T)$ and $(t(s), x(s)) \in Q$.  We say that $Q$
  satisfies the $t$-GCC if there exists $T>0$ such that $(Q,T)$ satisfies
  the $t$-GCC.

  The control time $T_0(Q,\Omega)$ is defined by
  \begin{equation*}
    T_0(Q,\Omega) = \inf \{ T>0\ \mid\ (Q,T)\ \text{satisfies the $t$-GCC} \} ,
  \end{equation*}
  with the agreement that $T_0(Q,\Omega)=+\infty$ if $Q$ does not
  satisfy the $t$-GCC.
\end{definition}

The $t$-GCC property of Definition \ref{def_tGCC} is a time-dependent
version of the usual GCC.

\begin{remark}
\label{rem: t-GCC}
Several remarks are in order.
\begin{enumerate}
\item 
The $t$-GCC assumption implies that the set $\mathcal O =
\cup_{t\in(0,T)}\omega(t)$ is a control domain that satisfies the
usual GCC for a time  $T>T_0(Q,\Omega)$.
\item 
  It is interesting to note that the control time $T_0(Q,\Omega)$ is
  not a continuous function of the domains, for any reasonable
  topology (see Remark \ref{rem_T0_notcont} below).
\item Observe that if $(Q,T)$ satisfies the $t$-GCC, 
a similar geometric condition may not occur if the time interval
$(0,T)$ is replaced by $(t_0, t_0 +T)$.  As the set $Q$  is not
a cylinder in general, by nature the $t$-GCC is not invariant under time
translation. 
\item Note that $Q$ cannot be chosen as an open set of $\R\times
  \Omega$ instead of $\R\times \ovl{\Omega}$. Consider indeed the case
  of a disk, if $Q$ is an open set of $\R\times
  \Omega$ then the ray that glides along the boundary never enters
  $Q$. This coincides with the so-called whispering gallery phenomenon.
\end{enumerate}
\end{remark}

\subsection{Main result}

\begin{theorem}\label{mainthm}
Let $Q$ be an open subset of $\R \times \ovl{\Omega}$ that satisfies the $t$-GCC. Let $T>T_0(Q,\Omega)$.
If $\partial\Omega\neq\emptyset$, we assume moreover that no
generalized bicharacteristic has a contact of infinite order with
$(0,T)\times\d\Omega$, that is, $G^\infty  = \emptyset$.
Then, the observability inequality \eqref{observability} holds.
\end{theorem}
\setcounter{thmbiss}{\arabic{theorem}}
Theorem \ref{mainthm} is proven in Section \ref{subsec_proofmainthm}.
By Proposition~\ref{prop: equivalence observability} we have the following result.
\begin{thmbis}\label{cor:mainthm}
Under the same assumptions as Theorem~\ref{mainthm} the observability inequality \eqref{observability-shift} holds.
\end{thmbis}

\begin{remark}
  \label{rem-bord}
\begin{enumerate}
  \item %\label{rem-bord1}
  In the case where $\d\Omega\neq\emptyset$, the assumption of the absence
  of any ray having a contact of infinite order with the
  boundary is classical (see \cite{BardosLebeauRauch}).  Note that
  this assumption is not useful if $M$ and $\d\Omega$ are analytic.
  This assumption is used in a crucial way in the proof of the theorem
  to ensure uniqueness of the generalized bicharacteristic flow, as
  stated in Proposition~\ref{prop: uniqueness bicharacteristic flow}.
  \item \label{rem-bord2}
  We have assumed here that, if $\d\Omega\neq\emptyset$, then
  $\d\Omega$ is smooth. The case where $\d\Omega$ is not smooth is
  open. Even the case where $\d\Omega$ is piecewise analytic is
  open. The problem is that, in that case, the generalized
  bicharacteristic flow is not well defined since there is no
  uniqueness of a bicharacteristic passing over a point.  This fact is
  illustrated on Figure~\ref{figcornerb} where a ray reflecting at
  some angle can split into two rays. However, it clearly follows from
  our proof that Theorem \ref{mainthm} is still valid if the domain
  $\Omega$ is such that this uniqueness property holds (like in the
  case of a rectangle).  In general, we conjecture that the conclusion
  of Theorem \ref{mainthm} holds true if \textit{all} generalized
  bicharacteristics meet $T^* Q$ within time $T$.  This would require
  however to extend the classical theory of propagation of
  singularities. Proving this fact is beyond
  of the scope of the present article. We may however assert here, in
  the present context, that the result of Theorem \ref{mainthm} is
  valid in any $d$-dimensional orthotope.
\end{enumerate}
\end{remark}

\begin{figure}[h]
\begin{center}
\subfigure[Right angle corner. \label{figcornera}]
{ \resizebox{3.8cm}{!}{
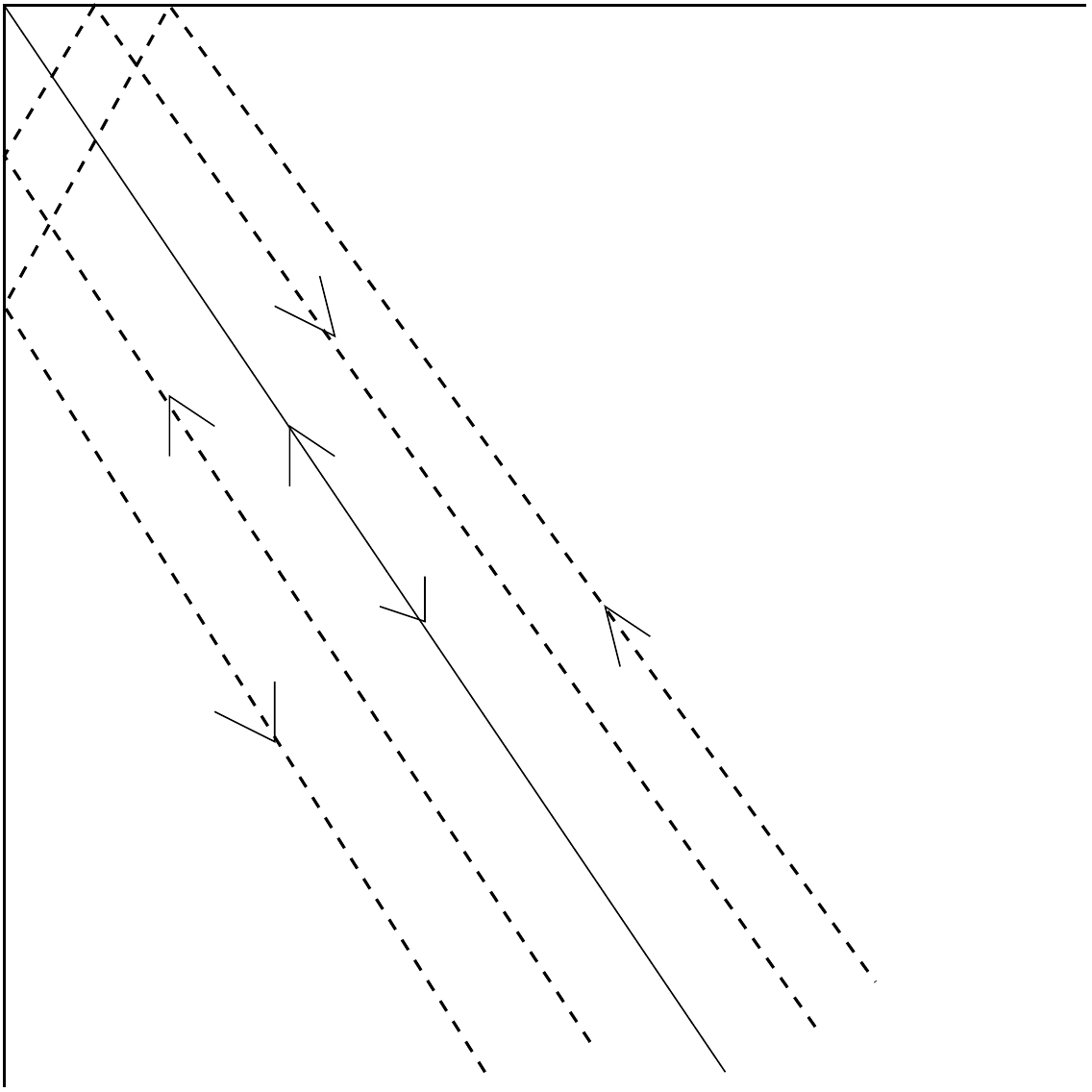
} } \quad 
\subfigure[General corner. \label{figcornerb}]
{ \resizebox{7.5cm}{!}{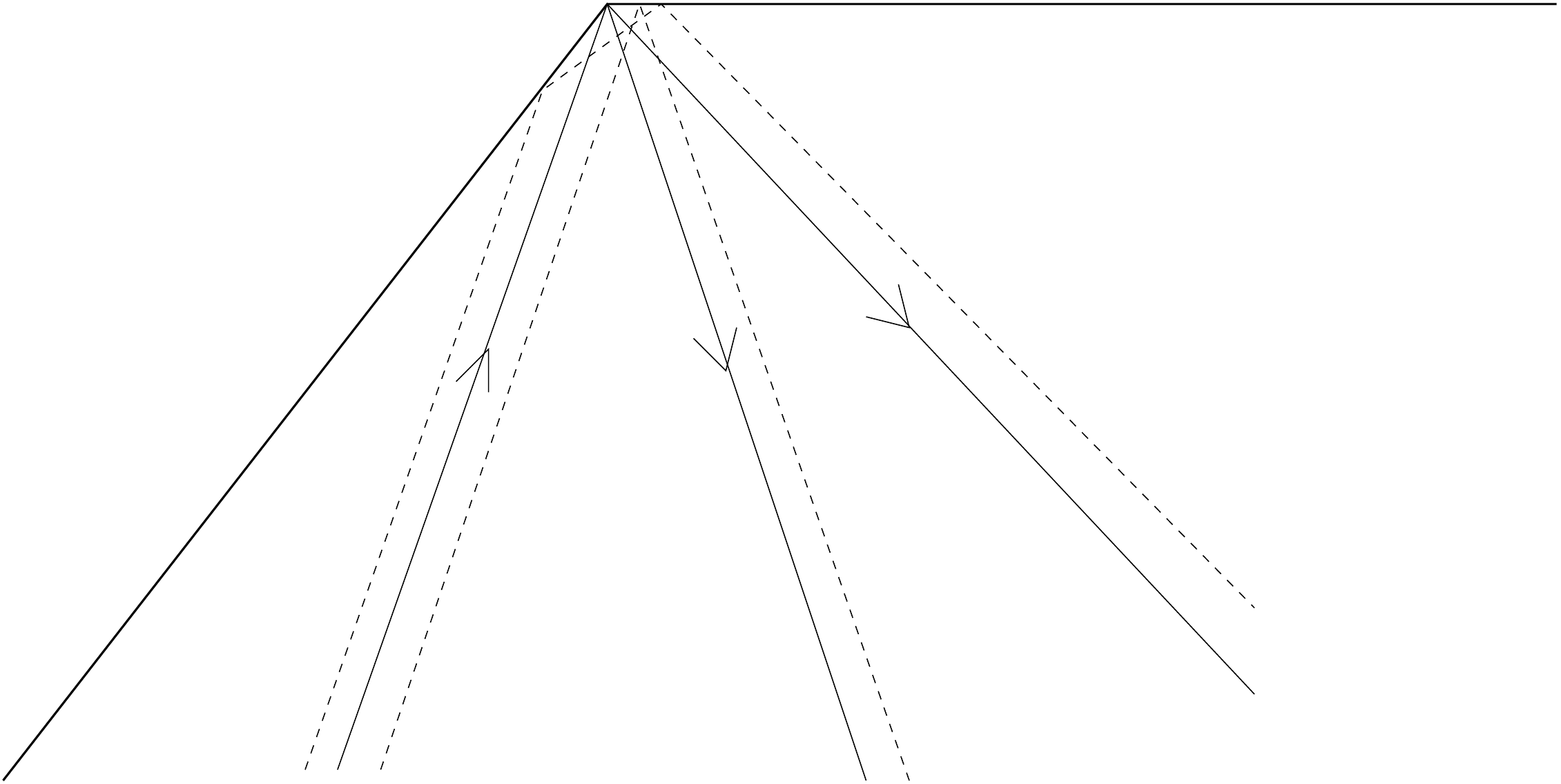} }
\caption{Reflection of generalized geodesics at some
  corner.}\label{fig_corner}
\end{center}
\end{figure}

\begin{remark}
  \label{rem: simplification BLR}
  In the case of a 1D wave equation with
  Dirichlet boundary conditions, the corresponding statement of
  Theorem~\ref{mainthm} is proven in \cite{Castro} by means of the
  D'Alembert formula.  The proof we provide in
  Section~\ref{subsec_proofmainthm}, is general and follows
  \cite{BardosLebeauRauch,BurqGerard}. In fact, as already mentioned
  in Section~\ref{sec: framework}, a key step of the
  approach of \cite{BardosLebeauRauch,BurqGerard} is simplified here.
  More precisely, the approach of \cite{BardosLebeauRauch,BurqGerard} consists of several steps. Firstly, a weaker version of the observability
  inequality is proven; in the present article, this is done in
  Lemma~\ref{lemma_weak_obs}. Secondly, the so-called \textit{set of
    invisible solutions} (defined by \eqref{def_NT}) is shown to be
  reduced to zero; in the present article, this is done in Lemma~\ref{lemNT}.
  Thirdly, the
  observability inequality is proven to hold by means of the result of the two previous steps.
  Our simplification with respect to \cite{BardosLebeauRauch,BurqGerard} lies in the second step. The argument is much
  shorter than the original one and, in the present analysis of a
  time varying observation region, it turns out to be crucial, as the
  more classical argument of \cite{BardosLebeauRauch,BurqGerard} cannot be applied.
\end{remark}

\subsection{Consequences}
\subsubsection{Controllability}
By the usual duality argument (Hilbert Uniqueness Method, see
\cite{lions2,lions}), we have the following equivalent result for the control of
the wave equation with a time-dependent control domain, based on the
observability inequality~\eqref{observability-shift} that follows from
Theorem~\ref{cor:mainthm}.

\begin{thmter}\label{cor1}
Let $Q$ be an open subset of $\R \times \ovl{\Omega}$ that satisfies the $t$-GCC. Let $T>T_0(Q,\Omega)$.
If $\partial\Omega\neq\emptyset$ we assume moreover that no
generalized bicharacteristic has a contact of infinite order with
$(0,T)\times\d\Omega$, that is, $G^\infty  = \emptyset$.
Setting $\omega(t) = \{ x\in\Omega\ \mid\  (t,x)\in Q\}$, we consider the wave equation with internal control
\begin{equation}\label{controlledwaveEq}
\d_t^2u-\triangle_g u = \chi_Q f,
\end{equation}
in $(0,T)\times\Omega$, with Dirichlet or Neumann boundary conditions
\eqref{condBC} whenever $\d\Omega\neq\emptyset$, and with
$f\in L^2((0,T)\times\Omega)$.  Then, the controlled equation
\eqref{controlledwaveEq} is exactly controllable in the space
$D(A^{1/2})\times X,$ meaning that, for all $(u^0,u^1)$ and
$(v^0,v^1)$ in $D(A^{1/2})\times X,$ there exists
$f\in L^2((0,T)\times\Omega)$ such that the corresponding solution of
\eqref{controlledwaveEq}, with $(u\ic,\d_t u\ic)=(u^0,u^1)$, satisfies
$(u_{|t=T},\d_t u_{|t=T})=(v^0,v^1)$.
\end{thmter}

\begin{remark}
  In the above result the control operator is $ f \mapsto \chi_Q f$. We
  could choose instead a control operator $f \mapsto   b(t,x) f$ with
  $b$ smooth and such that $Q =\{ (t,x) \in \R \times \ovl{\Omega}\ \mid \   b(t,x) >0\}$. Then with the same $t$-GCC we also have exact
controllability in this case. The equivalent observability inequality
is then 
\begin{equation*}
C \Norm{ (u^0,u^1)}{X\times D(A^{1/2})'}^2 \leq \Norm{b
  u}{L^2((0,T)\times\Omega)}^2 = \int_0^T \!\!\int_{\Omega} \vert b(t,x) u(t,x)\vert^2 \, dx_g\, dt.
\end{equation*}
\end{remark}

\subsubsection{Observability with few sensors}
We give another interesting consequence of Theorem \ref{mainthm}, in
connection with the very definition of the $t$-GCC property,
which can be particularly relevant in view of practical applications.

\begin{corollary}\label{cor_eco}
Let $Q\subset\R\times \ovl{\Omega}$ be an open subset with Lipschitz boundary and let $T>0$ be such that $(Q,T)$ satisfies the $t$-GCC.
Then, every open subset $\mathcal{V}$ of $[0,T]\times\ovl{\Omega}$ (for the
topology induced by $\R\times M$), containing $\d \left( Q\cap
  ([0,T]\times\ovl{\Omega}) \right)$, is such that $(\mathcal{V}, T)$
satisfies the $t$-GCC and, consequently, observability holds for such an
open subset.
\end{corollary}
\begin{proof}
  Let $\R \ni s \mapsto \comp \gamma(s)$ be a compressed generalized bicharacteristic with $t=
  t(s)$. As $(Q,T)$ satisfies the $t$-GCC, there exists $t_1 \in (0,T)$
  and $s_1 \in \R $ such that $t_1 = t(s_1)$ and $\comp \gamma(s_1) \in
  j(T^\ast (Q))$. Now, there are two cases.
  \begin{description}
    \item[Case 1.] There exists $s_2 \in \R$ such that $t_2 =
      t(s_2) \in (0,T)$ and $\comp \gamma(s_2) \notin
  j(T^\ast (Q))$. The continuity of $s \mapsto
  \comp \gamma(s)$ into  $\comp T^\ast Y \setminus E$, in particular
  of its projection on $\R \times \ovl{\Omega}$, then allows one to
  conclude that there exists $s_3 \in \R $ such that $t_3 =
      t(s_3) \in (0,T)$ and $\comp \gamma(s_3) \in
  j(T^\ast (\mathcal V))$.
\item[Case 2.] For all $s \in \R $ such that $t= t(s) \in
  (0,T)$ we have $\comp \gamma(s) \in  j(T^\ast (Q))$. Such a
  bicharacteristic is illustrated in Figures~\ref{fig_boundaryb} and \ref{fig_boundaryc}.
Then $s \mapsto
  \comp \gamma(s)$ enters $j(T^\ast (\mathcal W))$ for any \nhd $\mathcal W$ of
  $\{T\} \times \omega(T)$ (or $\{0\} \times \omega(0)$). Thus, there
  exists $s_2 \in \R $ such that $t_2 =
      t(s_2) \in (0,T)$  and $\comp \gamma(s_2) \in
 j( T^\ast (\mathcal V))$.
  \end{description}
\end{proof}

\begin{figure}[H]
\begin{center}
\subfigure[\label{fig_boundarya}]
{\resizebox{4.3cm}{!}{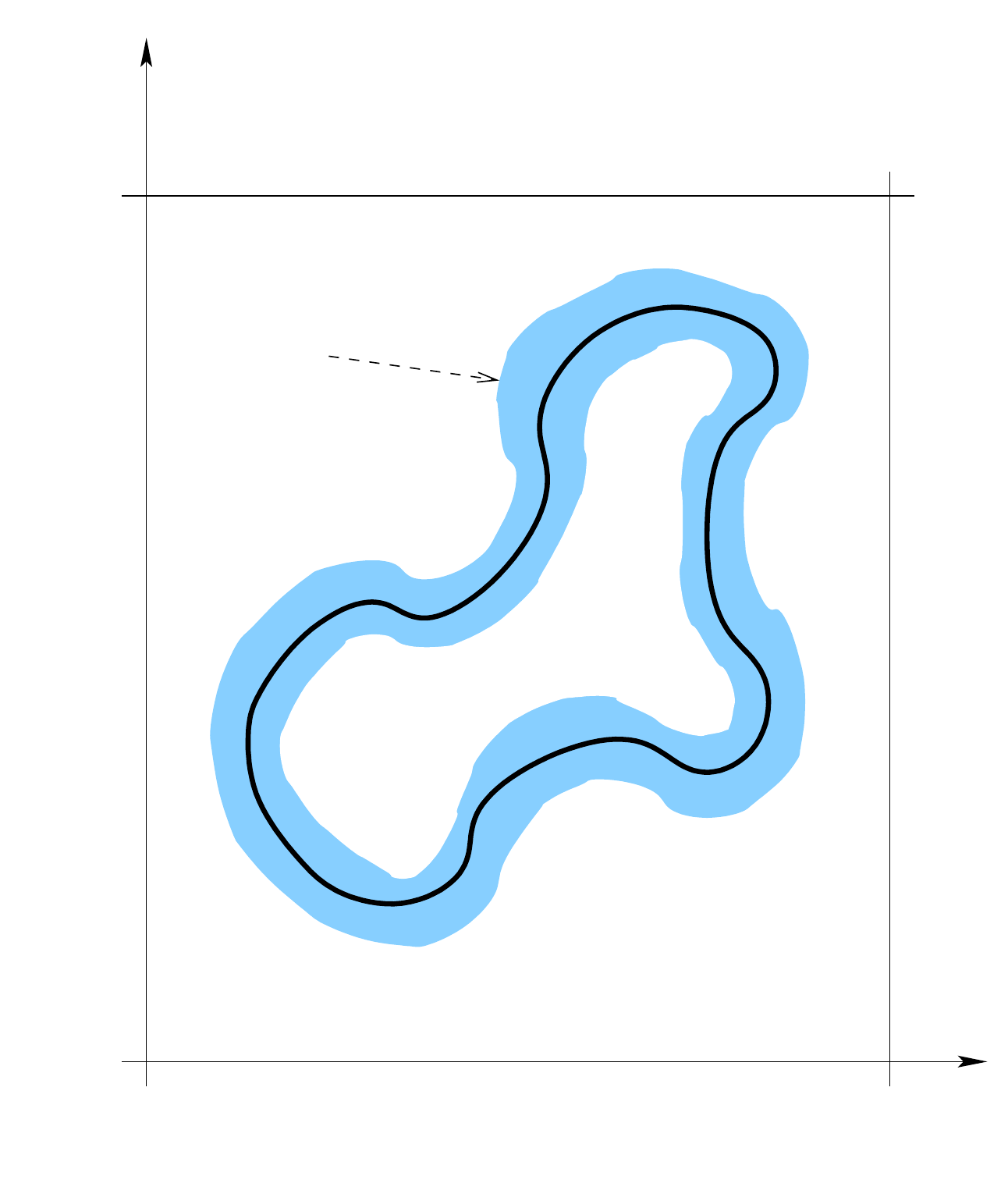}}\qquad 
\subfigure[\label{fig_boundaryb}]
{\resizebox{4.3cm}{!}{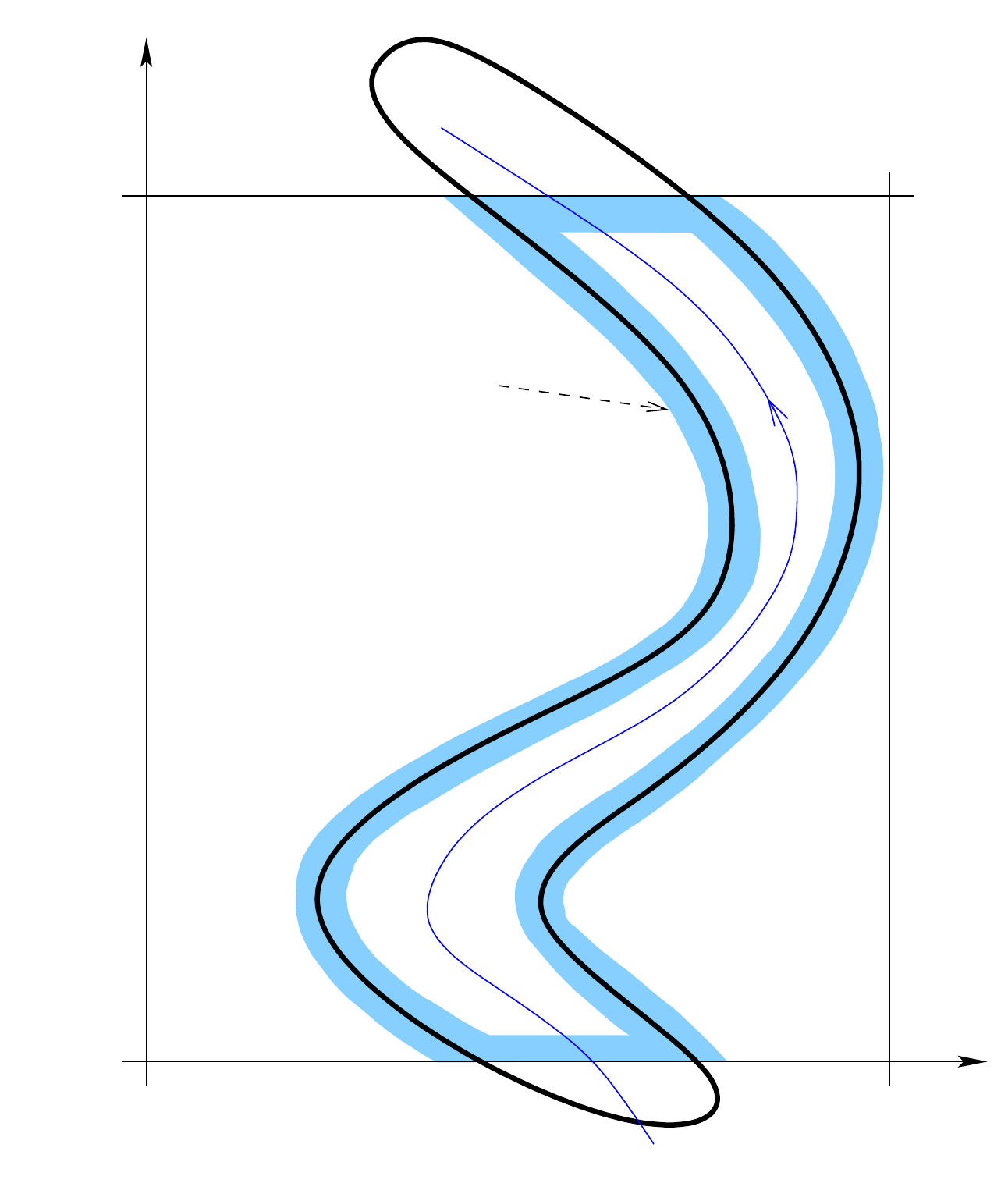}}\qquad 
\subfigure[\label{fig_boundaryc}]
{\resizebox{4.3cm}{!}{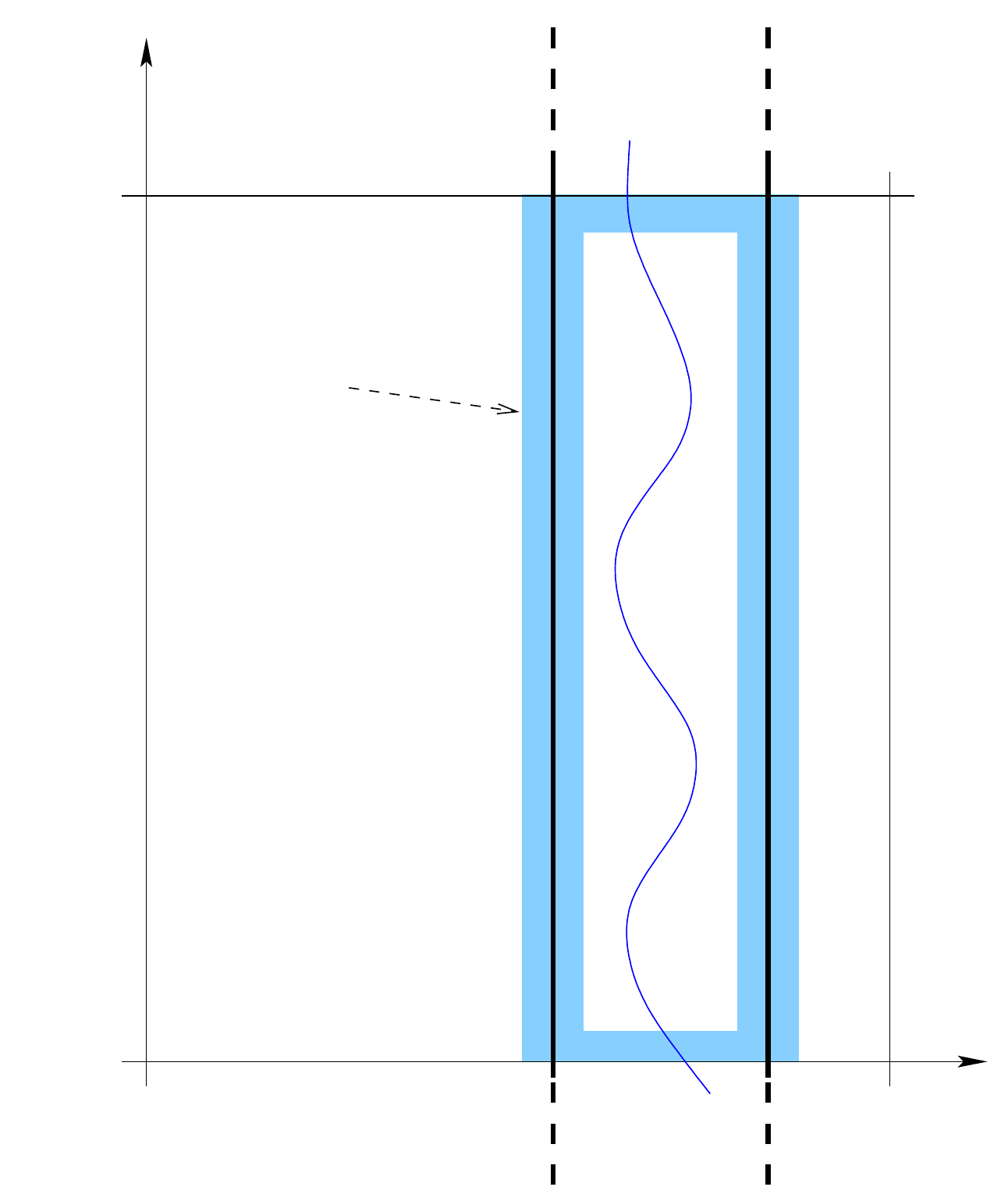}}
\caption{Neighborhood $\mathcal{V}$ of $\d \left( Q\cap
    ([0,T]\times\ovl{\Omega}) \right)$ in $[0,T]\times\ovl{\Omega}$
  (for the induced topology).
In \ref{fig_boundaryb} and \ref{fig_boundaryc}, a potential
bicharacteristic that remains in the interior of $Q$ is represented. 
}\label{fig_boundary}
\end{center}
\end{figure}

\begin{remark}$\phantom{-}$
\begin{enumerate}
\item 
The main interest of Corollary \ref{cor_eco} is that it allows one to take the open set $\mathcal{V}$ ``as small as possible", provided that it contains the boundary of $Q\cap ([0,T]\times\ovl{\Omega})$ (see Figure \ref{fig_boundary}).
As a practical consequence, only few sensors are needed to ensure the
observability property, or, by duality, the controllability
property, thus reducing the cost of an experiment.

Somehow, with an internal control we have $d+1$ degrees of freedom for
the control of $d$ variables, and this explains intuitively why the
choice of a ``thin'' open set $\mathcal V$ is possible. The above
corollary roughly states that control is still feasible by using only
$d$ degrees of freedom. In terms of the control domain, this means
that we only need a control domain that is any open neighborhood of a
set of Hausdorff dimension $d$. 
\item Observe that the proof of Corollary~\ref{cor_eco} and
  Figures~\ref{fig_boundaryb} and \ref{fig_boundaryc} show in fact
  that it suffices to choose $\mathcal V$ as the union of a \nhd of
  $\d Q
  \cap (0,T)\times \ovl{\Omega}$ and a \nhd of $Q \cap \{ t=0\}$ (or $Q \cap \{ t=T\}$).
\item 
Note that, if an open  subset $\omega$ of $\ovl{\Omega}$ satisfies the usual GCC, then a small neighborhood of $\d\omega$ does not satisfy necessarily the GCC. In contrast, when considering time-space control domains (i.e., subsets of $\R\times\ovl{\Omega}$), the situation is different. For instance, if $\omega$ satisfies the GCC, then any open subset of $[0,T]\times\ovl{\Omega}$, containing $\d \left( [0,T]\times\omega\right) = \left( [0,T]\times\d\omega \right) \cup \left( \{0\}\times\omega \right) \cup \left( \{T\}\times\omega \right) $, satisfies the $t$-GCC (see Figure \ref{fig_boundaryc}).
\item 
Note that $T(Q,\Omega) \leq \liminf T(\mathcal{V},\Omega)$ as
$\mathcal{V}$ shrinks to $\d \left( Q\cap ([0,T]\times\ovl{\Omega})
\right)$, and that equality may fail as there may exist
some bicharacteristics propagating inside $Q$ and not reaching
$\mathcal V$ for $t \in (t_1,t_2)$ for with $0< t_1 < t_2< T$ (see Figures
\ref{fig_boundaryb} and \ref{fig_boundaryc}).
\end{enumerate}
\end{remark}

\subsubsection{Stabilization}
Theorem \ref{mainthm} has the following consequence for wave equations
with a damping localized on a domain $Q$ that is time-periodic.

\begin{corollary}\label{cor_stab}
Let $Q$ be an open subset of $\R\times\ovl{\Omega}$, satisfying the $t$-GCC. Let
$T> T_0(Q,\Omega)$. If $\partial\Omega\neq\emptyset$, we assume moreover that no
generalized bicharacteristic has a contact of infinite order with
$(0,T)\times\d\Omega$, that is, $G^\infty  = \emptyset$.
Setting $\omega(t) = \{ x\in\Omega\ \mid\ (t,x)\in Q\}$, we assume that $\omega$ is $T$-periodic, that is, 
$\omega(t+T)=\omega(t)$,
for almost every $t\in(0,T)$.
We consider the wave equation with a local internal damping term, 
\begin{equation}\label{dampedwaveEq}
\d_t^2u-\triangle_g u + \chi_\omega \d_t u = 0,
\end{equation}
in $(0,T)\times\Omega$, with Dirichlet or Neumann boundary conditions \eqref{condBC} whenever $\d\Omega\neq\emptyset$.
Then, there exists $\mu\geq 0$ and $\nu >0$ such that any solution of \eqref{controlledwaveEq}, with $(u(0),\d_t u(0))\in D(A^{1/2})\times X$, satisfies
$$
E_0(u)(t) \leq \mu E_0(u)(0) e^{-\nu t},
$$
where we have set
$E_0(u)(t) = \frac{1}{2} \big( \Vert A^{1/2}u(t)\Vert_{L^2(\Omega)}^2 
+ \Vert \d_t u(t)\Vert_{L^2(\Omega)}^2 \big)$.
\end{corollary}
Corollary \ref{cor_stab} is proven in Section \ref{sec_proof_cor_stab}.

\section{Proofs}\label{sec_proof}

\subsection{Proof of Theorem \ref{mainthm}}\label{subsec_proofmainthm}
The proof follows the classical chain of arguments developed in
\cite{BardosLebeauRauch,BurqGerard}, with yet a simplification of one
of the key steps, as already pointed out in Remark~\ref{rem:
  simplification BLR}.  This simplification is a key element here. The
original proof scheme would not allow one to conclude in the case of a
time-dependent control domain.

For  a solution $u$ of \eqref{waveEq}-\eqref{condBC}, with $u_{|t=0} =
u^0 \in
D(A^{1/2})$ and $\d_t u_{|t=0} = u^1 \in X$,
we use the natural energy 
\begin{align}
  \label{eq: energy -1}
  E_{0}(u) (t) = \hf \big(\| u(t)\|_{D(A^{1/2})}^2 + \| \d_t u(t)\|_{X}^2\big).
\end{align}
In the proof, we use the fact
that this energy remains constant as time evolves, that is, 
\begin{align}
  \label{eq: energy preserved}
  E_{0}(u)(t) =  E_{0}(u)(0)  = \hf \big(\| u^0\|_{D(A^{1/2})}^2 + \|
  u^1\|_{X}^2\big) = \hf \Norm{(u^0,u^1)}{{D(A^{1/2})}\times X}^2,
\end{align}
and we shall simply write $E_{0}(u)$ at places.

We first achieve a weak version of the observability inequality.
\begin{lemma}\label{lemma_weak_obs}
There exists $C>0$ such that
\begin{equation}\label{weak_observability}
 C \Norm{(u^0,u^1)}{{D(A^{1/2})}\times X}^2
 \leq \Norm{\chi_Q \d_t u}{L^2((0,T)\times\Omega)}^2
 + \Norm{(u^0,u^1)}{X \times D(A^{1/2})'}^2,
\end{equation}
for all $(u^0,u^1)\in {D(A^{1/2})}\times X$, where $u$ is the corresponding solution of \eqref{waveEq}-\eqref{condBC} with $u\ic=u^0$ and $\d_t u\ic=u^1$.
\end{lemma}

\begin{remark}
With respect to the desired observability inequality
\eqref{observability}, the inequality \eqref{weak_observability}
exhibits a penalization term, $\Norm{(u^0,u^1)}{X \times D(A^{1/2})'}$, on the right-hand side. Note that the Sobolev
spaces under consideration have no importance, and instead of
$X \times D(A^{1/2})'$ we could as well have chosen $H^{1/2-s}(\Omega)\times H^{-1/2-s}(\Omega)$, for any $s>0$.
The key point lies in the fact that the space ${D(A^{1/2})}\times X$ is compactly embedded into $X \times D(A^{1/2})'$.
\end{remark}

\begin{proof}
We prove the result by contradiction. We assume that there exists
a sequence $(u^0_n,u^1_n)_{n\in\N}$ in ${D(A^{1/2})}\times X$ such that
\begin{align}
& \Norm{(u^0_n,u^1_n)}{{D(A^{1/2})}\times X}=1
\quad\forall n\in\N,
\label{e1} \\
& \Norm{(u^0_n,u^1_n)}{X \times D(A^{1/2})'}
\xrightarrow[n\rightarrow+\infty]{} 0, \label{e2} \\
\intertext{and}
& \Norm{\chi_Q \d_t u_n}{L^2((0,T)\times\Omega)}
\xrightarrow[n\rightarrow+\infty]{} 0, \label{e3}
\end{align}
where $u_n$ is the solution of \eqref{waveEq}-\eqref{condBC}
satisfying ${u_n}\ic=u_n^0$ and
$\partial_t {u_n}\ic=u_n^1$.
From \eqref{e1}, the sequence $(u^0_n,u^1_n)_{n\in\N}$ is bounded in
${D(A^{1/2})}\times X$, and using \eqref{e2} the only possible closure
point for the weak topology of ${D(A^{1/2})}\times X$ is
$(0,0)$. Therefore, the sequence $(u^0_n,u^1_n)_{n\in\N}$ converges to
$(0,0)$ for the weak topology of ${D(A^{1/2})}\times X$. By continuity
of the flow with respect to initial data, it follows that the sequence
$(u_n)_{n\in\N}$ of corresponding solutions converges to $0$ for the
weak topology of $H^1((0,T)\times\Omega)$; in particular, it is bounded.

Up to a subsequence (still denoted $(u_n)_{n\in\N}$ in what follows),
according to Proposition~\ref{eq: existence measure} in Appendix \ref{appendice}, there exists  a
microlocal defect measure $\mu$ on the cosphere quotient
space $S^\ast  \hat{\Sigma}$ introduced in Section~\ref{sec: Gk}, such that
\begin{align}
\label{eq: defect measure}
( R u_n , u_n) \xrightarrow[n\to+\infty]{} \langle \mu, \kappa(R)
\rangle,
\end{align}
for every $R \in \Psi^0(Y)$  with $\kappa(R)$ to  be understood as a
continuous function on $S^\ast  \hat{\Sigma}$.

It follows from \eqref{e3} that $\mu$ vanishes in
$j(T^*Q) \cap S^\ast \hat{\Sigma}$.  As is well known, the measure
$\mu$ is invariant\footnote{The theorem of propagation for measures is
  proven in \cite{Lebeau:96} for a damped wave equation with Dirichlet
  boundary condition.  We claim that the same proof applies with
  Neumann boundary condition. First, using the notations in
  \cite{Lebeau:96}, we still have $\mu_\partial=0$ in the Neumann
  case, where $\mu_\partial$ is defined by (A.18) in \cite{Lebeau:96}.
  Then, the proof of theorem A.1 in \cite{Lebeau:96} about the
  propagation of the measure still applies in the Neumann case.
  First, one can assume that the tangential operator $Q_0$ that
  appears in (A.28) satisfies $\partial_x Q_0\vert_{x=0}=0$, to assure
  that $Q_0 u$ still satisfies the Neumann boundary condition. Then
  inequality (A.29) in \cite{Lebeau:96} holds true as well, since the
  theorem of propagation of Melrose-Sj\"ostrand holds true in the
  Neumann case \cite{MelroseSjostrand1,MelroseSjostrand2}. Finally,
  estimate (A.33) of \cite{Lebeau:96} remains valid since the energy
  estimate holds true in the Neumann case.} under the compressed
generalized bicharacteristic flow \cite{Lebeau:96,BurqLebeau}. The
definition of this flow is recalled in Section~\ref{sec: generalized
  bichar}.

The $t$-GCC assumption for $Q$ then implies that
$\mu$ vanishes identically (see
\cite{BardosLebeauRauch,BurqGerard}). This precisely means that
$(u_n)_{n\in \N}$ strongly converges to $0$ in $H^1((0,T)\times
\Omega)$.

Now, we let $0< t_1< t_2 < T$. The above strong convergence implies that
\begin{equation*}
  %\label{eq: energy to 0}
\int_{t_1}^{t_2} E_{0}(u_n)(t) \, dt \xrightarrow[n\rightarrow+\infty]{} 0,
\end{equation*}
with the energy $E_{0}$ defined in \eqref{eq: energy -1}. 
As this energy is preserved (with respect to time $t$), this implies that 
\begin{align*}
  E_{0}(u_n)(0) = \hf \big(\| u_n^0\|_{D(A^{1/2})}^2 + \|
  u_n^1\|_{X}^2\big) \xrightarrow[n\rightarrow+\infty]{} 0, 
\end{align*}
yielding a contradiction. 
\end{proof}

We define the set of \emph{invisible solutions} as
\begin{multline}\label{def_NT}
N_T = \{ v \in H^1((0,T)\times\Omega) \ \mid\ v \ \text{is a solution of
  \eqref{waveEq}-\eqref{condBC},} \\\text{with $v\ic \in {D(A^{1/2})}$, $\d_t v\ic  \in X$
 \ and}\ \chi_Q \d_t v=0 \}, 
\end{multline}
equipped with the norm $\| v\|_{N_T}^2  = \| v\ic\|_{D(A^{1/2})}^2 + \|
  \d_t v\ic \|_{X}^2$.
Clearly, $N_T$ is closed.

\begin{lemma}\label{lemNT}
We have $N_T=\{0\}$.
\end{lemma}

In other words, due the $t$-GCC assumption, there is no nontrivial invisible solution.

\begin{proof}
  First, the $t$-GCC assumption combined with the propagation of
  singularities along the generalized bicharacteristic flow (see
  \cite[Theorem 24.5.3]{Hoermander:V3}) implies that all elements of
  $N_T$ are smooth functions on $(0,T)\times\Omega$, up to the
  boundary. In particular, if $v\in N_T$ then $\d_t v\in N_T$.

  Second, we remark that, since the operator $\d_t^2 - \triangle_g$ is
  time-independent (as well as the boundary condition), the space $N_T$ is invariant under the action of
  the operator $\d_t$.

  Third, applying the weak observability inequality of Lemma
  \ref{lemma_weak_obs} gives
  \begin{equation*}
    C  \| v\|_{N_T}^2  
    = C \Norm{(v\ic,\d_t v\ic )}{{D(A^{1/2})}\times X}
    \leq\Norm{(v\ic,\d_t v\ic )}{X \times D(A^{1/2})'},
   \end{equation*}
for every $v \in N_T$. Since ${D(A^{1/2})}\times(D(A^{1/2}))'$ is compactly embedded into $X \times D(A^{1/2})'$, this implies that the unit ball of $N_T$ is compact and thus $N_T$ is finite dimensional.

We are now in a position to prove the lemma. The proof goes by contradiction. Let us assume that $N_T\neq\{0\}$. The operator $\d_t:N_T\rightarrow N_T$ has at least one (complex) eigenvalue $\lambda$, associated with an eigenfunction $v\in N_T\setminus\{0\}$. Since $\d_t v=\lambda v$, it follows that $v(t,x)=e^{\lambda t} w(x)$,
and since $(\d_{t}^2-\triangle_g)v=0$ we obtain $(\lambda^2-\triangle_g)w=0$. 
Note that $ \lambda\neq 0$ (in the Neumann case, we have $w\in L^2_0(\Omega)$).
Now, take any $t\in(0,T)$ such that $\omega(t)=\{x\in\Omega\ \mid\  (t,x)\in Q\}\neq\emptyset$. Since
$\chi_Q\d_t v=0$ and thus $\chi_Q v=0$, it follows that $w=0$ on the open set $\omega(t)$. 
By elliptic unique continuation we then infer that $w=0$ on the whole $\Omega$, and hence $v=0$. This is a contradiction.
\end{proof}

Let us finally derive the observability inequality
\eqref{observability}. To this aim, the compact term on the
right-hand-side of \eqref{weak_observability} must be removed.  We
argue again by contradiction, assuming that there exists a sequence
$(u^0_n,u^1_n)_{n\in\N}$ in ${D(A^{1/2})}\times X$ such that
\begin{align}
& \Vert (u^0_n,u^1_n) \Vert_{{D(A^{1/2})}\times X}=1,\quad\forall n\in\N, \label{ee1} \\
& \Norm{\chi_Q\d_t u_n}{L^2((0,T)\times\Omega)} \xrightarrow[n\rightarrow+\infty]{} 0, \label{ee2}
\end{align}
where $u_n$ is the solution of \eqref{waveEq}-\eqref{condBC} such that
${u_n}\ic=u^0_n$ and $\d_t {u_n}\ic=u^1_n$. From \eqref{ee1}, the
sequence $(u^0_n,u^1_n)_{n\in\N}$ is bounded in ${D(A^{1/2})}\times X$,
and therefore, extracting if necessary a subsequence, it converges to
some $(u^0,u^1)\in {D(A^{1/2})}\times X$ for the weak topology.
Let $u$ be the solution of \eqref{waveEq}-\eqref{condBC} such that
$u\ic=u^0$ and $\d_t u\ic=u^1$. Then,
$\chi_Q\d_t u_n \to \chi_Q\d_t u$ weakly in $L^2((0,T)\times\Omega)$ implying
$$
\Norm{\chi_Q \d_t u}{L^2((0,T)\times\Omega)} \leq
\liminf_{n\rightarrow+\infty} \Norm{\chi_Q\d_t u_n}{L^2((0,T)\times\Omega)} = 0,
$$
and hence $u \in N_T$. It follows from Lemma \ref{lemNT} that
$u=0$. In particular, we have then $(u^0,u^1)=(0,0)$ and hence
$(u^0_n,u^1_n)_{n\in\N}$ converges to $(0,0)$ for the weak topology of
${D(A^{1/2})}\times X$, and thus, by compact embedding, for the strong topology of $X \times D(A^{1/2})'$. Applying the weak observability inequality \eqref{weak_observability} raises a contradiction. This concludes the proof of Theorem \ref{mainthm}.
\hfill \qedsymbol\endproof

\subsection{Proof of Proposition~\ref{prop: equivalence observability}}\label{sec_rem_obs_equiv}
  First, we assume that the observability inequality~\eqref{observability} holds.
  Let $v$ be a solution of \eqref{waveEq}-\eqref{condBC}, with initial
  conditions $(v^0,v^1)\in X\times D(A^{1/2})'$. We set
  $u=\int_0^t v(s) d s  - A^{-1} v^1$. Then $\d_t u=v$ and we have
  $u\ic=u^0=-A^{-1}v^1\in D(A^{1/2})$, and $\d_t u\ic=u^1=v^0\in
  X$. Moreover, we have
$$
\d_t^2 u(t) = \d_t v(t) = \int_0^t \d_t^2 v(s) d s+ \d_tv(0) 
= -  \int_0^t Av (s) d s  + v^1 = -A u(t).
$$
Since $v=\d_tu$ and
$
\Norm{ (u^0,u^1)}{D(A^{1/2})\times X} = \Norm{ (A^{-1}v^1,v^0)}{D(A^{1/2})\times X} = \Norm{ (v^0,v^1)}{X\times D(A^{1/2})'},
$
applying the observability inequality~\eqref{observability} to $u$, we obtain \eqref{observability-shift}.

Second,  we assume that the observability
inequality~\eqref{observability-shift} holds. 
Let $u$ be a solution of \eqref{waveEq}-\eqref{condBC}, with initial
conditions $(u^0,u^1)\in D(A^{1/2})\times X$. We set $v=\d_tu$. Then
$v$ is clearly a solution of \eqref{waveEq}-\eqref{condBC}, with
$v\ic=v^0=u^1\in X$ and $\d_tv
\ic=v^1=\d_t^2u\ic=-Au\ic = - A u^0 \in D(A^{1/2})'$.
Since
$
\Norm{ (v^0,v^1)}{X\times D(A^{1/2})'} = \Norm{ (u^1,Au^0)}{X\times D(A^{1/2})'}
= \Norm{ (u^0,u^1)}{D(A^{1/2})\times X},$
applying  the observability inequality~\eqref{observability-shift}  to $v=\d_tu$, we obtain \eqref{observability}.
\hfill \qedsymbol \endproof

\subsection{Proof of Corollary \ref{cor_stab}}\label{sec_proof_cor_stab}
It is proven in \cite{Haraux} that a second-order linear equation with
(bounded) damping has the exponential energy decay property, if and
only if the corresponding conservative linear equation is observable.
The extension to our framework is straightforward. We however give a
proof of Corollary \ref{cor_stab} for completeness.

By Theorem~\ref{mainthm}, there exists
$C_0>0$ such that 
\begin{equation}\label{ineqobsconservative}
  C_0 \Big( \Vert A^{1/2}\phi\ic\Vert_{L^2(\Omega)}^2 + \Vert\d_t\phi\ic\Vert_{L^2(\Omega)}^2\Big)
  \leq \int_0^S \!\! \Vert \d_t\phi \Vert_{L^2(\omega(t))}^2\,
  dt,\qquad S = \ell T, \ \ell \in \N^\ast, 
\end{equation}
for $\phi$ solution of
$\d_t^2\phi = \triangle_g\phi,
$
with  $(\phi\ic,\d_t\phi\ic)\in D(A^{1/2})\times X$.

Let now $u$ be a solution of \eqref{dampedwaveEq} with $(u\ic,\d_t
u\ic)\in D(A^{1/2})\times X$. Let us prove that we have an exponential decay for
its energy. We consider $\phi$ as above with the following initial conditions $\phi\ic=u\ic$ and $\d_t\phi\ic=\d_t u\ic$. Then, setting $\theta=u-\phi$, we have
\begin{align}
  \label{eq: wave eq theta}
\d_t^2\theta  -  \triangle_g\theta =  \chi_\omega \d_t u, \qquad
  \theta\ic=0, \ \d_t\theta\ic=0.
\end{align}
Observe that $\d_t u \in L^2(\R \times \Omega)$ yielding $\theta \in
\Con^0(\R; D(A^{1/2})) \cap \Con^1(\R; X)$.

Were the \rhs of \eqref{eq: wave eq theta} to be replaced by $f$ in $H^1(\R \times \Omega)$, we would have $\theta \in
\Con^0(\R; D(A)) \cap \Con^1(\R; D(A^{1/2}))$. Recalling the
definition of $E_0$ in the statement of Corollary~\ref{cor_stab}, we
would find
\begin{align*}
  \frac{d}{d t}E_0 (\theta)(t) 
  = \langle\d_t\theta(t),A\theta(t)+\d_t^2 \theta(t)\rangle_{L^2(\Omega)}
  = \langle\d_t\theta(t), f \rangle_{L^2(\Omega)}.
\end{align*}
Continuity with respect to $f$ and a density argument
then yield $\frac{d}{d
  t}E_0 (\theta)(t) = \langle\d_t\theta(t), \chi_\omega \d_t u\rangle_{L^2(\Omega)}$. 
With two integrations with respect to $t \in (0,S)$, using that $E_0
(\theta)(0)=0$, we obtain, by the Fubini theorem,
$
\int_0^S E_0(\theta)(t)\, dt = \int_0^S (S-t)  \int_{\omega(t)} \d_t\theta(t,x) \d_t u(t,x)\, dx_g\, dt .
$
With the Young inequality, we have
%\begin{align*}
$\int_0^S  \!\! E_0(\theta)(t)\, dt
\leq  S^2 \int_0^S\!\! \Vert \d_t u\Vert_{L^2(\omega(t))}^2 \, dt 
+ \frac{1}{4} \int_0^S \!\! \Vert \d_t \theta
\Vert_{L^2(\Omega)}^2\, dt .
$
%\end{align*}
With the definition of $E_0(\theta)(t)$, we then infer that
\begin{equation}
  \label{ineq_theta_u}
  \int_0^S \!\! \Vert \d_t \theta\Vert_{L^2(\Omega)}^2\, dt
  \leq  4 S^2 \int_0^S\!\!  \Vert \d_t u\Vert_{L^2(\omega(t))}^2 \, dt .
\end{equation}
Now, since $\phi = u-\theta$, we have $\Vert \d_t\phi \Vert_{L^2(\omega(t))}^2 \leq 2\Vert \d_t
u\Vert_{L^2(\omega(t))}^2 + 2\Vert \d_t \theta\Vert_{L^2(\Omega)}^2$,
yielding, 
using \eqref{ineq_theta_u},
\begin{equation}
\label{eq: ineq_theta_u2}
\int_0^S \!\! \Vert \d_t\phi \Vert_{L^2(\omega(t))}^2 \, dt
\leq (2+8S^2)\int_0^S\!\!  \Vert \d_t u\Vert_{L^2(\omega(t))}^2 \, dt
.
\end{equation}
Arguing as above, we have $\frac{d}{dt}E_0(u)(t) = -
\Vert \d_t u(t,x)\Vert_{L^2(\omega(t))}^2$.
Using this property, inequalities \eqref{ineqobsconservative} and \eqref{eq: ineq_theta_u2}, and the fact that $\phi\ic=u\ic$ and
$\d_t\phi\ic=\d_t u\ic$, we deduce that $C_0
E_0(u)(0) = C_0
E_0(\phi)(0) \leq (2+8S^2)  (E_0(u)(0)-E_0(u)(S))$, or rather  $E_0
(u)(S) \leq \big(1 -C_0/ (2+8S^2)\big)
E_0(u)(0)$. For $S$ chosen \suff large, that is, for $\ell \in \N^\ast$
chosen \suff large, we thus have $E_0
(u)(S) \leq \alpha E_0(u)(0)$ with $0 < \alpha < 1$. 

Since $\omega$ is $T$-periodic and thus $S$-periodic, the above
reasoning can be done on any interval $(kS,(k+1)S)$, yielding
$E_0(u)((k+1)S)\leq \alpha E_0(u)(kS)$, for every $k\in\N$.  Hence, we
obtain 
$E_0 (u)(kS)\leq \alpha^k E_0(u)(0) $.

For every $t\in [kS,(k+1)S)$, noting that $k=\left[t/S\right]> t/S-1$,
and that $\log(\alpha)<0$, it follows that
$\alpha^k<\frac{1}{\alpha}\exp(\frac{\ln(\alpha)}{S} t)$ and hence
$E_0(u)(t)\leq E_0(u)(S) \leq \mu\exp(-\nu t) E_0(u)(0)$, for some positive
constants $\mu$ and $\nu$ that are independent of $u$.
\hfill \qedsymbol \endproof

\section{Some examples and counter-examples}\label{sec_examples}
In the forthcoming examples, we shall consider several geometries in
which the observation (or control) domain
$\omega(t)=\{x\in\Omega\ \mid\ (t,x)\in Q\}$ is the rigid displacement in
$\Omega$ of a fixed domain, with velocity $v$. Then the resulting
observability property depends on the value of $v$ with respect to the
wave speed.

In all our examples, in the presence of a boundary we shall consider
Dirichlet boundary conditions. In that case, generalized
bicharacteristics behave as described in Section~\ref{sec: generalized
  bichar}.
We recall that, if parametrized by time $t$, the projections of the
generalized bicharacteristics on the base manifold travel at speed one. 

\subsection{In dimension one}
\label{sec: ex dim 1}
We consider $M=\R$ (Euclidean) and $\Omega=(0,1)$. The rays have a speed equal to $1$.
We set $I=(0,a)$, for some fixed $a\in(0,1)$, and we assume that the control domain $\omega(t)$ is equal to the translation of the interval $I$ with fixed speed $v>0$. 
We have, then, $\omega(t)=(vt,vt+a)$ as long as $t\in(0,(1-a)/v)$.
When $\omega(t)$ touches the boundary, we assume that it is ``reflected" after a time-delay $\delta\geq 0$, according to the following rule:
if $(1-a)/v\leq t\leq (1-a)/v+\delta$ then $\omega(t) = (1-a,1)$. For larger times $t\geq (1-a)/v+\delta$ (and before the second reflection), the set $\omega(t)$ moves in the opposite direction with the same speed (see Figure \ref{fig1d}).

\begin{figure}[h]
\begin{center}
\subfigure[Case $v <1$ and $\delta\geq 0$. \label{fig1d-a}]
{\resizebox{4cm}{!}{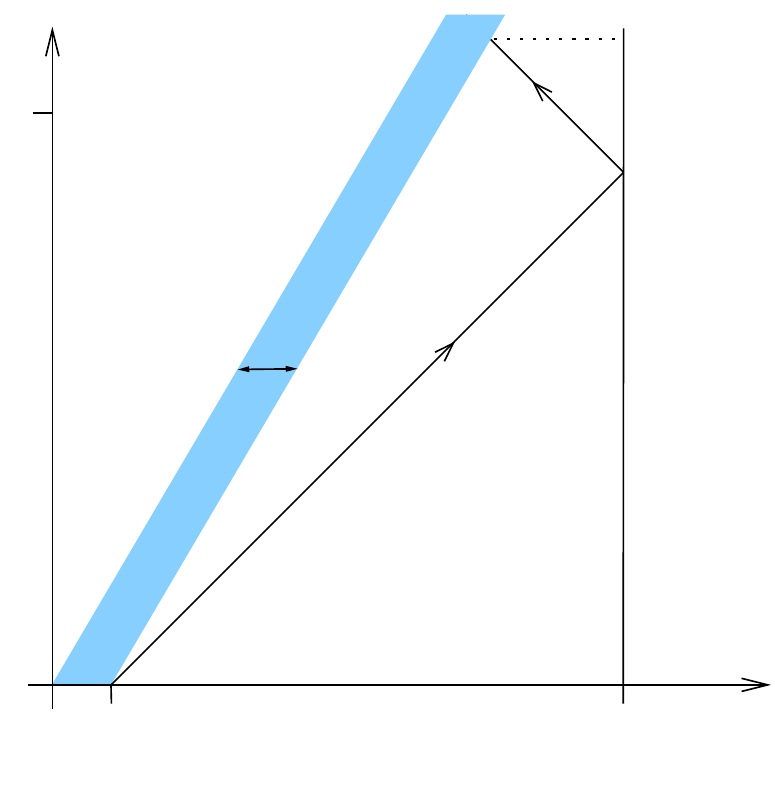}}\qquad \qquad
\subfigure[Case $v =1$ and $\delta>0$. \label{fig1d-d}]
{\resizebox{4cm}{!}{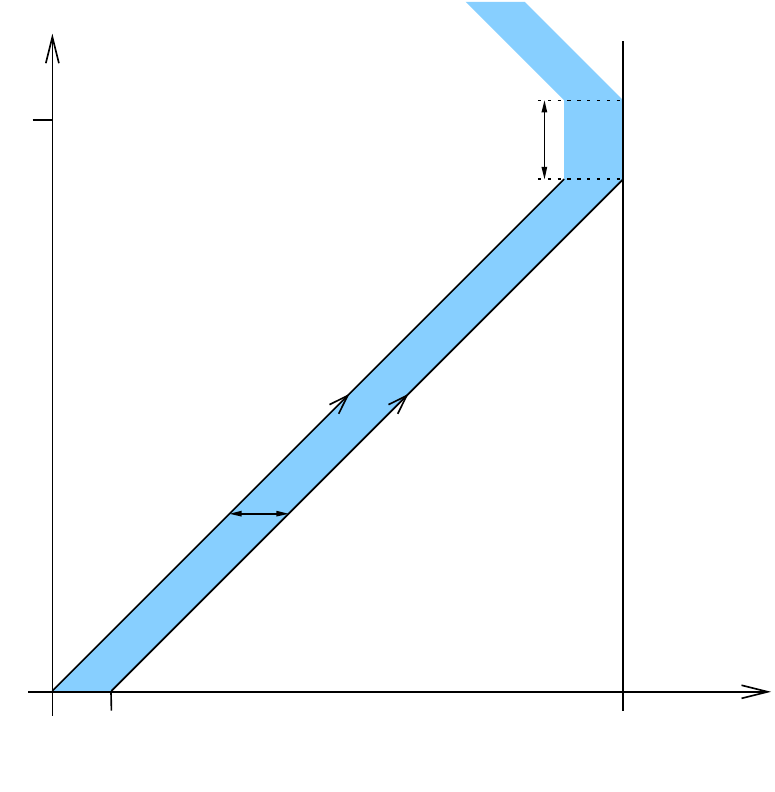}}
\\
\subfigure[Case $v \geq 1$ and $\delta=0$. \label{fig1d-c}]
{\resizebox{4cm}{!}{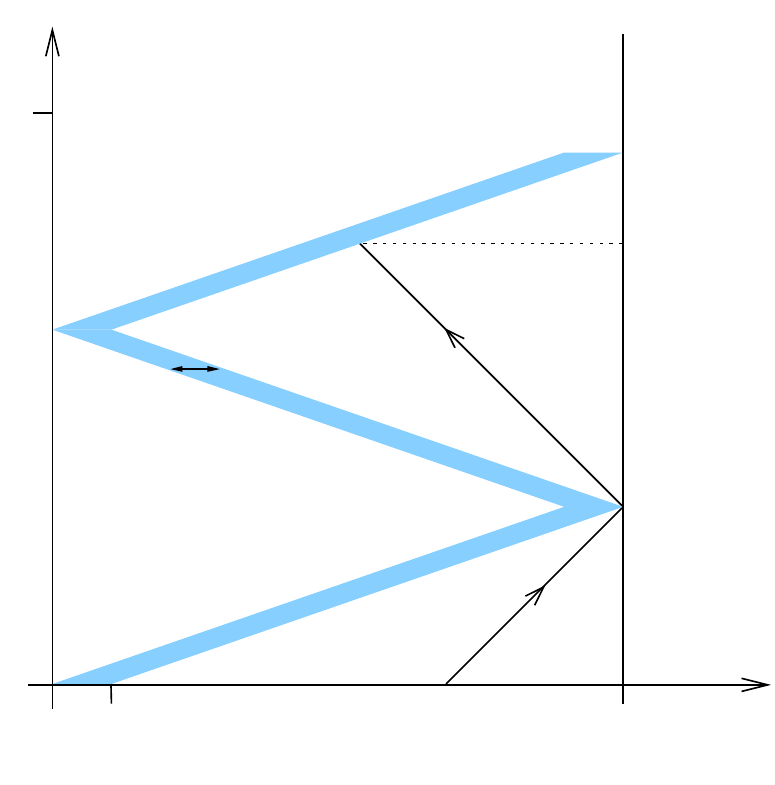}}\qquad \qquad
\subfigure[Case $v >1$ and $\delta>0$. \label{fig1d-b}]
{\resizebox{4cm}{!}{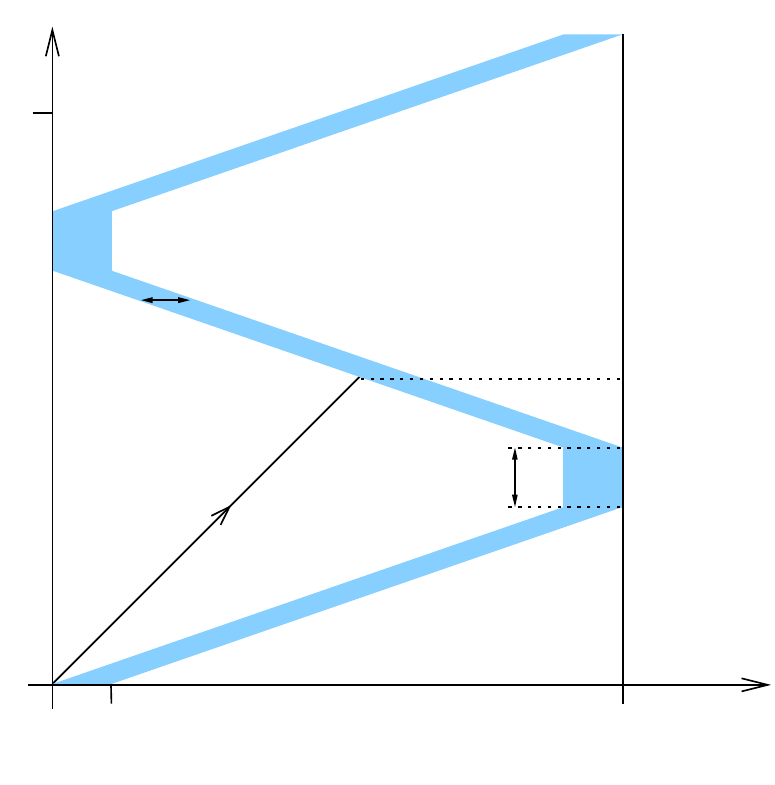}}
\end{center}
\caption{Time-varying domains in dimension one.}\label{fig1d}
\end{figure}

This simple example is of interest as it exhibits that the control
time depends on the value of the velocity $v$ with respect to the wave
speed (which is equal to $1$ here). We denote by  $T_0(v,a,\delta)$ the control time.
With simple computations (see also Figure~\ref{fig1d}), we establish that
\begin{equation*}
T_0(v,a,\delta) = \begin{cases}
2(1-a)/(1+v) & \textrm{if}\ 0\leq v < 1\ \text{and}\
\delta\geq 0 ,\\
1-a & \text{if}\  v=1 \ \text{and}\ \delta>0,\\
(1-a)(3v+1)/ \big(v(1+v)\big) & \text{if}\ v\geq 1\ \text{and}\ \delta=0,\\
\big( 2(1-a)+v\delta\big) (1+v) & \text{if}\ v> 1\ \text{and}\ \delta>0.
\end{cases}
\end{equation*}

\begin{remark}
\label{rem_T0_notcont}
Note that the control time $T_0(v,a,\delta)$ is discontinuous in $v$
and $\delta$. The control time is not continuous with respect to the
domain $Q$ as already mentioned in Remark~\ref{rem: t-GCC}.
\end{remark}

\subsection{A moving domain on a sphere}
\label{sec: ex sphere}
Let $M=\Omega=S^2$, the unit sphere of $\R^3$, be endowed with the
metric induced by the Euclidean metric of $\R^3$.  Let us consider
spherical coordinates $(\theta,\varphi)$ on $M$, in which $\varphi=0$
represents the horizontal plane (latitude zero), and $\theta$ is the
angle describing the longitude along the equator. Let $a\in(0,2\pi)$
and $\eps\in(0,\pi/2)$ be arbitrary.  For $v>0$, we set
$$
\omega(t) = \{ (\theta,\varphi)\ \ \mid\ \vert\varphi\vert<\eps,\ vt<\theta<vt+a\},
$$
for every $t\in\R$. The set $\omega(t)$ is a spherical square drawn on
the unit sphere, with angular length equal to $2\eps$ in
latitude, and $a$ in longitude, and moving along the equator with
speed equal to $v$ (see Figure \ref{figsphere}). We denote by
$T_0(v,a,\eps)$ the control time as defined in Section~\ref{sec: t gcc}. 

\begin{figure}[h]
\begin{center}
\resizebox{6cm}{!}{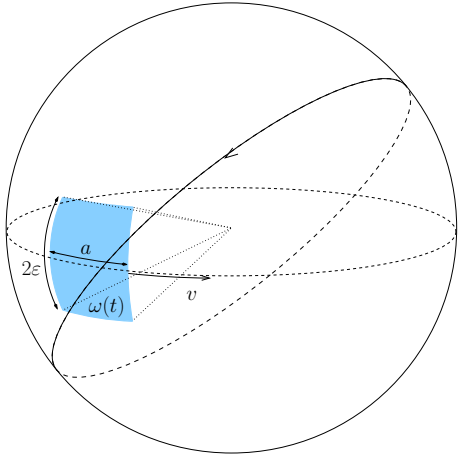}
\end{center}
\caption{Time-varying domain on the unit sphere and a typical ray (great circle).}\label{figsphere}
\end{figure}

For this example, an important fact is the following: every (geodesic)
ray on the sphere propagates at speed one along a great circle, with
half-period $\pi$. We thus have a simple description of all possible rays.
Note that, as the radius is one, the speed coincides with the angular speed.

\begin{proposition}\label{prop_sphere}
Let $a\in(0,2\pi)$ and $\eps\in(0,1)$ be arbitrary. Then $T_0(v,a,\eps)<+\infty$ except for a finite number of critical speeds $v>0$. Moreover:
\begin{itemize}
\item $T_0(v,a,\eps)\sim\frac{\pi-a}{v}$ as $v\rightarrow 0$;
\item If $v > v_1= (2 \pi - a + 2 \eps) / (2\eps)$ then $T_0(v,a,\eps) < \infty$. If $v\rightarrow+\infty$ then $T_0(v,a,\eps)\rightarrow \pi-2\eps$.
\end{itemize}
Besides, if $v\in\Q$, then there exist $a_0>0$ and $\eps_0>0$ such that $T_0(v,a,\eps)=+\infty$ for every $a\in(0,a_0)$ and every $\eps\in(0,\eps_0)$.
\end{proposition}

Obtaining an analytic expression of $T_0$ as a function of $(v,a,\eps)$ seems to be very difficult.

Note that the asymptotics above still make sense if either $a>0$ or
$\eps>0$ are small. This shows that we can realize the observability property
with a subset of arbitrary small  Lebesgue measure (compare with Corollary
\ref{cor_eco}).

\begin{proof}[Proof of Proposition \ref{prop_sphere}.]
  First, we observe the following. Consider a ray propagating along
  the equator (with angular speed $1$), in the same direction as
  $\omega(t)$.  If $v=1$, depending on its initial condition, this ray
  either never enters $\omega(t)$ or remains in it for all time.
  Hence, $T_0(v,a,\eps)=+\infty$. In contrast, if $v\neq 1$,
  then, such a ray enters $\omega(t)$ for a time
  $0\leq t < (2\pi)/\vert v-1\vert$, as
  $(2\pi)/\vert v-1\vert > (2\pi-a)/\vert v-1\vert$.

Second, we treat the cases $v$ large and $v$ small, and we compute the
asymptotics of $T_0(v,a,\eps)$ (in the argument, both $a$ and $\eps$ are
kept fixed). 
\begin{description}
\item[Case  $\bld{v}$ small.] If $2 \pi v < a$, then every ray goes
  full circle in a time shorter than that it takes for the domain to
  travel the distance $a$. It is then  clear that  every ray will have
  met $\omega(t)$ as soon as $\omega(t)$ has travelled halfway along
  the equator (up to the thickness of $\omega(t)$ and the travel time
  of the ray itself).
  In other words, we then have  $\frac{\pi-a}{v} \leq T_0(v,a,\eps)
  \leq \frac{\pi-a}{v} + 2 \pi$ for $v < v_0 = a/(2\pi)$. 

\item[Case $\bld{v}$ large.]
If $v$ grows to infinity, then the situation becomes intuitively as if
we have a static control domain forming  a strip of constant width
$\eps>0$ around the equator. For such a strip,  the control
time is $\pi-2\eps$. More precisely,  let us assume $(2\pi-a+2\eps) /v < 2 \eps$.
Every ray entering the region $\{ |\varphi| < \eps\}$, spends a time
at least equal to $2 \eps$ in this region. 
At worst, the control domain will have to travel the distance $2 \pi -a
+ 2 \eps$ to ``catch'' this ray (going full circle and more than the
longitudinal distance travelled by the ray itself).   The condition $v >v_1 = 
(2\pi-a+2\eps)/ (2 \eps)$ thus implies that all rays enter  the moving open domain
$\omega(t)$ within time $\pi - 2 \eps + 2(\pi+\eps) /v$. Hence,
$\pi - 2 \eps\leq T_0(v,a,\eps) \leq \pi - 2 \eps + 2(\pi+\eps) /v$. 
\end{description}

Third, we consider the case $v_0 \leq v \leq v_1$.  To get some
intuition, we consider, in a first step, that $a$ and $\eps$ are
both very small, and thus consider $\omega(t)$ as point moving along
the equator.  According to the first observation we made above, let us
consider a ray propagating along a great circle that is transversal to
the equator. It meets the equator at times $t_k = t_0 + k \pi$,
$k \in \Z$, for some $t_0$.  If $v$ is irrational then the set of
positions of the ``points'' $\omega(t_k)$, given by
$x(t_k)=\cos(vt_k)$ and $y(t)=\sin(vt_k)$ in the plane $(x,y)$
containing the equator, is dense in the equator. Adding some thickness
to $\omega(t)$, that is having $a>0$ and $\eps>0$, we find that every
ray encounters the moving open set $\omega(t)$ in a finite time if $v$
is irrational. By a compactness argument we then obtain
$T_0(v,a,\eps)< \infty$ if $v$ is irrational.

Fourth, considering again that $a>0$ and $\eps>0$ are both very small,
we shall now see that there do exist rays, transversal to the equator,
that never meet the moving ``point'' $\omega(t)$ whenever $v\in\Q$.
Writing $v=p/q$ with $p$ and $q$ positive integers, the set of points
reached by $(\cos(vt_k),\sin(vt_k))$ at times $t_k = t_0 + k\pi$, with
$k\in\Z$, is finite. The following
lemma yields a more precise statement.
\begin{lemma}\label{lem_regpolyg}
Let $p$ and $q$ be two coprime integers. We have
\begin{equation*}
\Big\{ \frac{kp}{q}\pi\mod 2\pi \ \ \mid\  k=1,\ldots,2q\Big\}
=
\begin{cases}
\displaystyle\Big\{ \frac{k}{q}\pi\ \ \mid\  k=1,\dots,2q\Big\} & \textrm{if $p$ is odd}, \\
\displaystyle\Big\{ 2\frac{k}{q}\pi\ \ \mid\  k=1,\dots,q\Big\} & \textrm{if $p$ is even (and $q$ odd)}.
\end{cases}
\end{equation*}
\end{lemma}
Thus, if $v=p/q$, with $p$ and $q$ coprime integers, the points
$(\cos(vt_k),\sin(vt_k))$ form the vertices of a regular polygon in
the disk. There are exactly $2q$ (\resp $q$) such vertices if $p$ is
odd (\resp even). In this situation, it is always possible to find a
ray transversal to the equator that never meets this set of vertices.
Now, this phenomenon persists in the case $a>0$ and $\eps>0$ if both
are chosen \suff small. We have thus proven that, given
$v\in\Q \cap [v_0, v_1]$, there exist $0< a_0< 2 \pi$ and
$0 < \eps_0< \pi/2$ such that $T_0(v,a,\eps)=+\infty$ for all
$a\in(0,a_0)$ and $\eps\in(0,\eps_0)$. Note also that if $a> 2\pi/q$
and $\eps>0$ then every ray meets $\omega(t)$ in some finite time.  By
a compactness argument we then obtain $T_0(v,a,\eps)< \infty$. From
that last observation, we infer that, given $a>0$ and $\eps>0$ fixed,
the set of rational velocities $v\in (v_0,v_1)\cap\Q$ for which
$T_0(v,a,\eps)=+\infty$ is finite.
\end{proof}

\begin{proof}[Proof of Lemma \ref{lem_regpolyg}.]
We note that 
$\big\{ kp \pi / q \mod 2\pi\ \ \mid\   k=1,\ldots,2q \big\} = \big\{  kp
\pi / {q} \mod 2\pi\ \ \mid\   k\in\Z\big\}$.
It thus suffices to prove the following two statements:
\begin{enumerate}

\item for $p$ even: $\forall k'\in \{1,\ldots,q\}, \ \exists k\in\Z, \
  \exists m\in\Z$ such that $2k'=kp+2mq$.
\item for $p$ odd: $\forall k'\in \{1,\ldots,2q\}, \ \exists
  k\in\Z, \ \exists m\in\Z$ such that $k'=kp+2mq$.
\end{enumerate}
Since $p$ and $q$ are coprime, there exists $(a,b)\in\Z^2$ such that
$ap+bq=1$. Moreover, if $(a,b)$ is a solution of that diophantine
equation, then all other solutions are given by $(a+qn,b-pn)$, with
$n\in\Z$. Multiplying by $2k'$, we infer that $2k'=2k'ap+2k'bq$, and the first 
statement above follows. For the second statement, we note that, if
$p$ is odd, then, changing $b$ into $b-pn$ if necessary, we may assume
that $b$ is even, say $b = 2 b'$. Then, multiplying by $k'$, we infer
that $k'=k'ap+2k'b'q$, and the second statement follows.
\end{proof}
Before moving on to the next example, we stress again that the
peculiarity of the present example (unit sphere) is that all
 rays are periodic, with the same period $2 \pi$. The study
of other Zoll
manifolds would be of interest. The
situation turns out to be drastically different in the case of a disk,
due to ``secular effects" implying a precession phenomenon, as we are
now going to describe.

\subsection{A moving domain near the boundary of the unit disk}\label{sec_disk}
Let $M=\R^2$ (Euclidean) and let
$\Omega=\{ (x,y)\in\R^2\ \mid\ x^2+y^2< 1\}$ be the unit disk.  Let
$a\in(0,2\pi)$ and $\eps\in(0,1)$ be arbitrary. We set, in polar
coordinates,
$$
\omega(t) = \{ (r,\theta)\in [0,1]\times\R\ \ \mid\  1-\eps< r < 1,\ vt<\theta<vt+a \},
$$
for every $t\in \R$. The time-dependent set $\omega(t)$ moves at constant angular speed $v$, anticlockwise, along the boundary of the disk (see Figure \ref{fig: moving domain on the disk}). Its radial length is $\eps$ and its angular length is $a$.

\begin{figure}[h]
\begin{center}
\resizebox{7cm}{!}{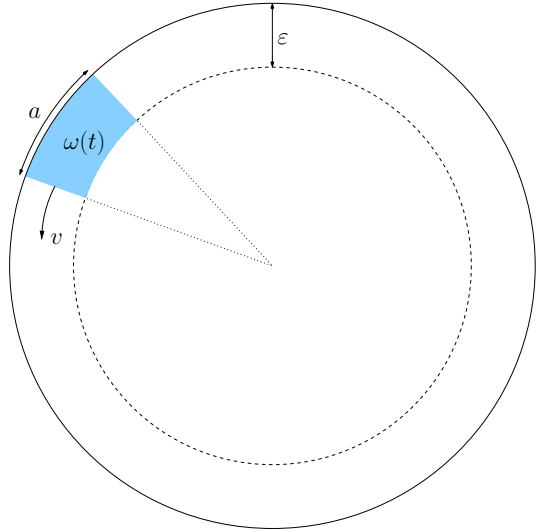}
\end{center}
\caption{A time-varying domain on the unit disk}
\label{fig: moving domain on the disk}
\end{figure}

\begin{proposition}\label{prop_disk}
The following properties hold:
\begin{enumerate}
\item Let $a\in(0,2\pi)$ and $\eps\in(0,1)$ be arbitrary. We have  $T_0(v,a,\eps)<+\infty$, for every $v>  v_0 = (2 \pi +
2\eps-a)/ (2 \eps)$, and we have $T_0(v,a,\eps)\sim 2-2\eps$ as $v\rightarrow
  +\infty$.
\item If there exists $n\in\N\setminus\{0,1\}$ such that
  $v\sin\frac{\pi}{n}\in\pi\Q$, then there exist $a_0\in (0, 2\pi)$ and $\eps_0\in(0,1)$ such that $T_0(v,a,\eps)=+\infty$ for all $a\in(0,a_0)$ and $\eps\in(0,\eps_0)$.

\item For every $v\geq 1$, for every $a\in(0,2\pi)$, there exists $\eps_0>0$ such that $T_0(v,a,\eps)=+\infty$, for every $\eps\in(0,\eps_0)$.
\item For every $v\geq 0$ and $a \in(0,\pi)$, there exists $\eps_0\in (0,1)$ such that $T_0(v,a,\eps)=+\infty$, for every $\eps\in(0,\eps_0)$.
\end{enumerate}
\end{proposition}
As for the case of the sphere presented in Section~\ref{sec: ex sphere}, obtaining an analytic expression of
$T_0$ as a function of $(v,a,\eps)$ seems a difficult task.

The fact that $T_0(v,a,\eps)=+\infty$ provided that $a>0$ and $\eps>0$ are
chosen \suff small  is in strong contrast with the case of the
sphere. This is due to the fact that, in the disk, the structure of
the rays is much more complex: there are large families of
periodic and almost-periodic rays.  The ray drawn in Figure
\ref{figdiskrays} produces some sort of ``secular effect", itself
implying a precession whose speed can be tuned to coincide with the
speed $v$ of $\omega(t)$, provided $v\geq 1$. We
shall use this property in the proof.

 We stress that for the third property we do not need to
assume that $a$ is small. Actually, $a$ is any element of
$(0,2\pi)$. If $a$ is close to $2\pi$, then $\omega(t)$ is almost a
ring located at the boundary, moving with angular speed $v$, with a
``hole". As this hole moves around with speed $v$ there is a ray
that periodically hits the boundary and reflects from it exactly at the hole position.

\begin{proof}[Proof of Proposition \ref{prop_disk}.]
For $a\in (0,2 \pi)$ and $\eps\in (0,1)$ fixed, if $v$ is very large,
then the situation gets close to  that of a static control domain which is a
ring of width $\eps$, located at the boundary of the disk. For this
static domain the control time is $2-2\eps$. In fact, all rays enter
the region $\Omega_\eps= \{ 1-\eps < r < 1\}$ and the shortest time spent there is
$2 \eps$. During such time the angular distance travelled by the ray
is less than $2 \eps$. Hence, if $(2 \pi + 2\eps-a) /v < 2 \eps$, one
knows for sure that the ray will be ``caught''  by the moving open set
$\omega(t)$ before it leaves $\Omega_\eps$. Thus, for $v > v_0 = (2 \pi +
2\eps-a)/ (2 \eps)$ all rays enter $\omega(t)$ in finite time. Moreover, we
have $2 - 2\eps \leq  T_0(v,a,\eps)\leq 2 - 2\eps + (2 \pi + 2\eps)
/v$. This yields the announced asymptotics for $T_0(v,a,\eps)$. 

\medskip

Let us now investigate the three cases where $T_0(v,a,\eps)=+\infty$ as
stated in the proposition.  For the sake of intuition, it is simpler to first assume that
$\eps>0$ and $a>0$ are very small, and hence, that $\omega(t)$ is
close to being a point moving along the boundary of the disk, given by
$(\cos(vt),\sin(vt))$.

Let us then consider, as illustrated in Figure~\ref{fig: periodic
  ray}, {\em periodic} rays propagating ``anticlockwise" in the disk
(with speed equal to $1$), and reflecting at the boundary of the disk
according to Section~\ref{sec: generalized bichar}, that is, according
to geometrical optics. The trajectory of such rays forms a regular
polygon with vertices at the boundary of the unit disk. Let $n\geq 2$
be the number of vertices. For $n=2$, the ray travels along a diameter
of the disk, and passes through the origin; it is $4$-periodic. For
$n=3$, the trajectory of the ray forms an equilateral triangle
centered at the origin; etc. The length of an edge of such a regular
polygon with $n\geq 2$ vertices is equal to $2\sin\frac{\pi}{n}$.
This means that there exists $t_0 \in \R$ such that this ray reaches
the boundary at times $t_k = t_0 + 2k \sin\frac{\pi}{n}$.  Hence, if
$2v\sin\frac{\pi}{n}=\frac{2p\pi}{q}$ with $p$ and $q$ positive
integers, then the moving point $(\cos(vt),\sin(vt))$, taken at times $t_k$
ranges over a
finite number of points of $\d \Omega$. Therefore, there
exists a periodic ray with $n$ vertices never meeting
$\omega(t)$. This property remains clearly true for values of $a>0$ and
of $\eps>0$ chosen \suff small. This show the second statement of the proposition.

\begin{figure}[t]
\begin{center}
\resizebox{6cm}{!}{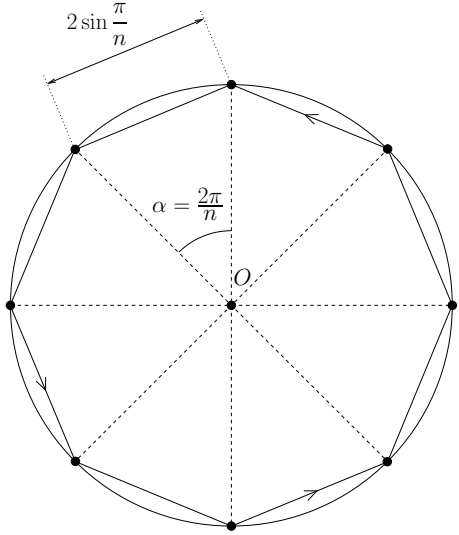}
\end{center}
\caption{A periodic ray yielding a regular polygon}
\label{fig: periodic ray}
\end{figure}

\medskip

Let us now consider a ray propagating in the disk, as
drawn in Figures~\ref{figdiskraysa} and \ref{figdiskraysb}, and reflecting at the boundary at
 consecutive points $P_k$, $k \in \N$. Denote by $O$ the center of the disk,
and by $\alpha$ the oriented angle $\widehat{P_0 O P_1}$. If $0<
\alpha < \pi$, then the ray appears to be going ``anticlockwise''
as in Figure~\ref{figdiskraysa}; if $\alpha=\pi$ then the ray bounces back
and forth on a diameter of the disk; if $\pi < \alpha < 2 \pi$,  then the ray appears to be going ``clockwise''
as in Figure~\ref{figdiskraysb}.  
\begin{figure}[h]
\begin{center}
\hspace*{-1cm}
\subfigure[``anticlockwise'', $0 <\alpha\leq \pi$. \label{figdiskraysa}]
{ \resizebox{8cm}{!}{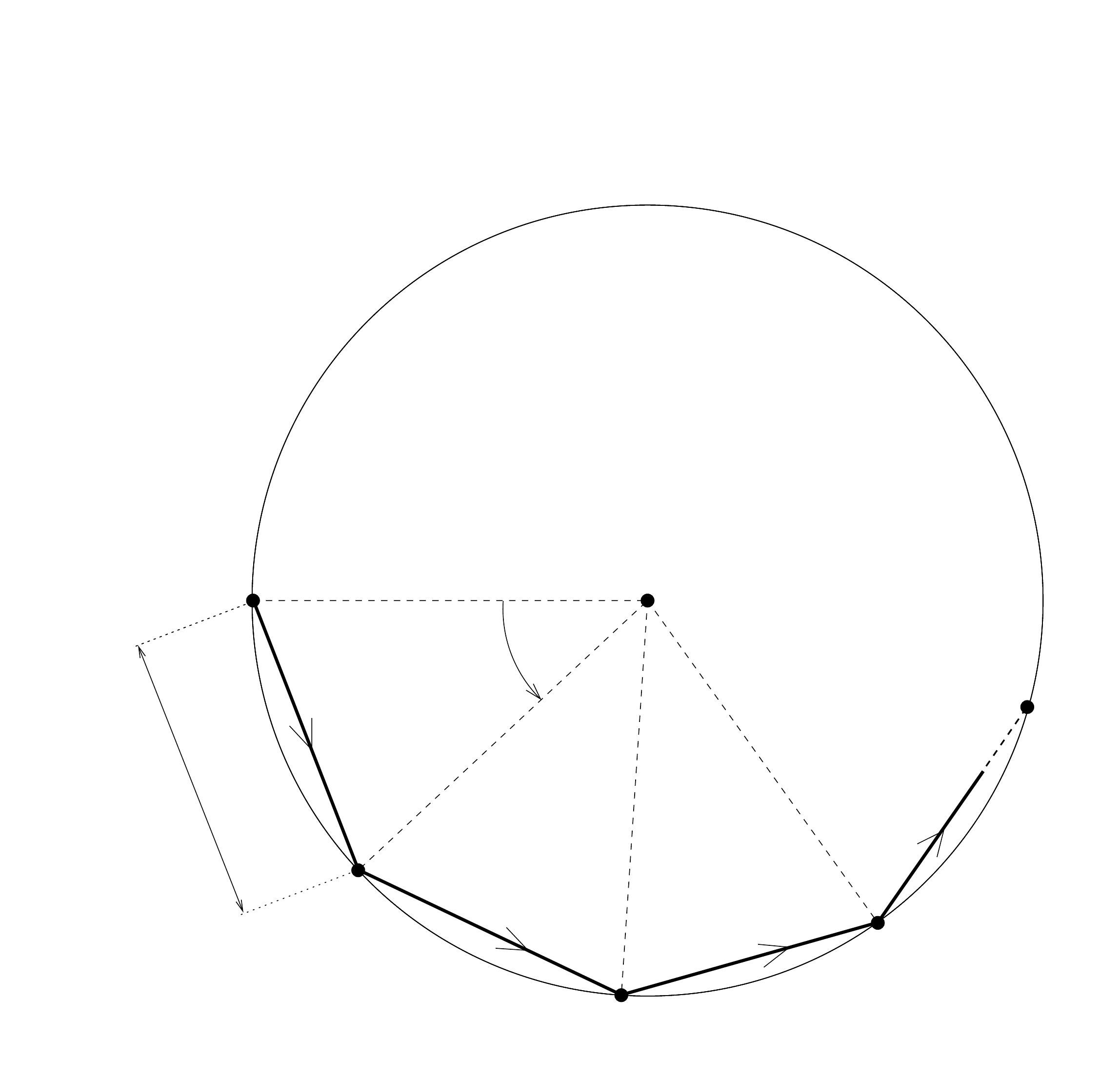} }\hspace*{-0.35cm}
\subfigure[``Clockwise'', $\pi\leq \alpha< 2 \pi$. \label{figdiskraysb}]
{ \resizebox{8cm}{!}{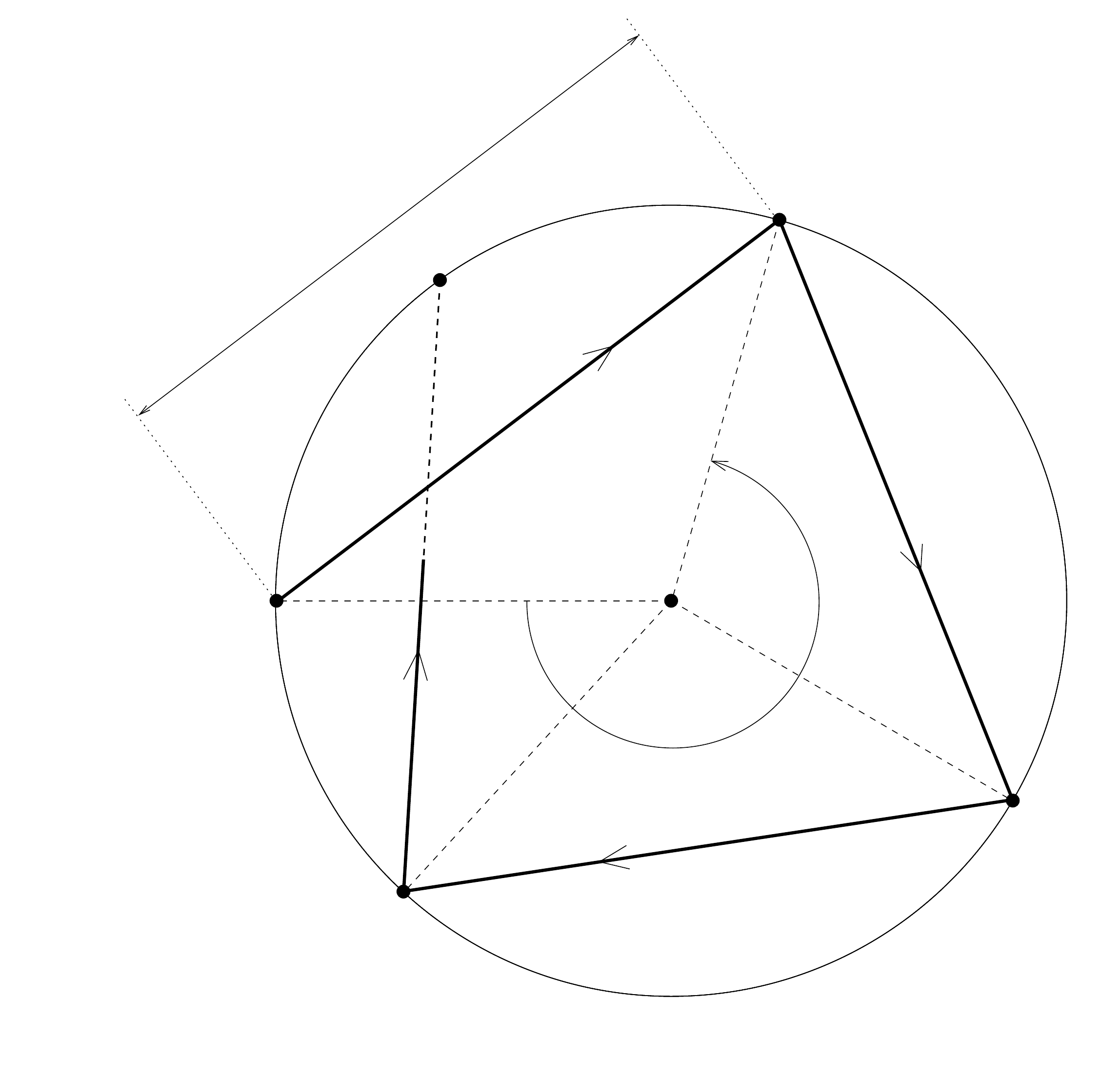} }
\end{center}
\caption{Rays propagating ``anticlockwise'' and ``clockwise''.}\label{figdiskrays}
\end{figure}
In any case, the distance $P_0P_1$, and more generally
$P_kP_{k+1}$, is equal to $2\sin(\alpha/2)$. Since the speed of the ray
is equal to $1$, the ray starting from $P_0$ at time $t=0$ reaches the
point $P_1$ at time $2\sin(\alpha/2)$, the point $P_k$ at time $t_k = 2k\sin(\alpha/2)$, etc.
Let $t\mapsto P(t)$ be the curve propagating {\em anticlockwise} along the
unit circle, with constant angular speed, passing exactly through the
points $P_k$ at time $t_k$. Its angular speed is then 
$$
w_P(\alpha) = \frac{\alpha}{2\sin(\alpha/2)},
$$
and we call it the \textit{precession speed}. This is the speed at
which the discrete points $P_k$ propagate ``anticlockwise" along the
unit circle.  Now, if the set $\omega(t)$ has the angular speed
$v=w_P(\alpha)$ (for some $\alpha\in(0,2\pi)$), then there exists
rays, as in Figures \ref{figdiskraysa} or \ref{figdiskraysb}, that never meet
$\omega(t)$, if $a \in (0,2\pi)$, provided that $\eps>0$ is chosen \suff small.  Since the
function $\alpha\mapsto w_P(\alpha)$ is monotone increasing from
$(0,2\pi)$ to $(1,\infty)$, it follows that, for every
$v\in (1,\infty)$, there exists $\alpha\in(0,2\pi)$ such that
$v=w_P(\alpha)$, and therefore $T_0(v,a,\eps)=+\infty$ provided that
$\eps$ is chosen \suff small. For $v =1$, there exists a gliding ray that
never meets $\omega(t)$. This can be seen as the limiting case
$\alpha \to 0$, as  rays can be  concentrated near the gliding ray. We thus
have  proven the third statement of the proposition.

Now, still working with the configurations drawn in Figure
\ref{figdiskraysb}, for $\alpha \in (\pi, 2 \pi)$, let $P(t)$ be the curve propagating {\em
  anticlockwise} along the unit circle, with constant angular speed,
passing successively through $P_0$ at time $0$, through $P_2$ at
time $4 \sin(\alpha/2)$, and $P_{2k}$ at time $2k \sin(\alpha/2)$.
Its angular  speed is given by 
$$
w_P(\alpha) = \frac{2\alpha- 2\pi}{4 \sin(\alpha/2)} = \frac{\alpha- \pi}{2 \sin(\alpha/2)}.
$$
The function $\alpha\mapsto w_P(\alpha)$ is monotone increasing from $(\pi,2\pi)$ to $(0,+\infty)$.
If the set $\omega(t)$, with $a \in (0,\pi)$ and $\eps>0$ small, is initially (at time $0$)
located between the points $P_0$ and $P_0'$ its diametrically opposite point, and if $v=w_P(\alpha)$, then the ray drawn in
Figure~\ref{figdiskraysb} never meets $\omega(t)$. This is illustrated
in Figure~\ref{figdiskrays3}. We have thus proven
the last statement of the proposition.
\end{proof}

\begin{remark}
$\phantom{-}$
\begin{enumerate}

\item
It is interesting to note that, even for domains such that $a$ is
close to $2\pi$, the $t-GCC$ property fails if $v>1$ and if $\eps$ is
chosen too small. This example is striking, because in that case, if the control domain were static, then it would satisfy the usual GCC (this is true as soon as $a>\pi$).
This example shows that, when considering a control domain satisfying the GCC, then, when making it move, the $t$-GCC property may fail.
However, this example is a domain moving faster than the actual
wave speed. This is rather non physical.

\begin{figure}[H]
\begin{center}
\subfigure[$t=0$. \label{figdiskrays3a}]
{ \resizebox{6cm}{!}{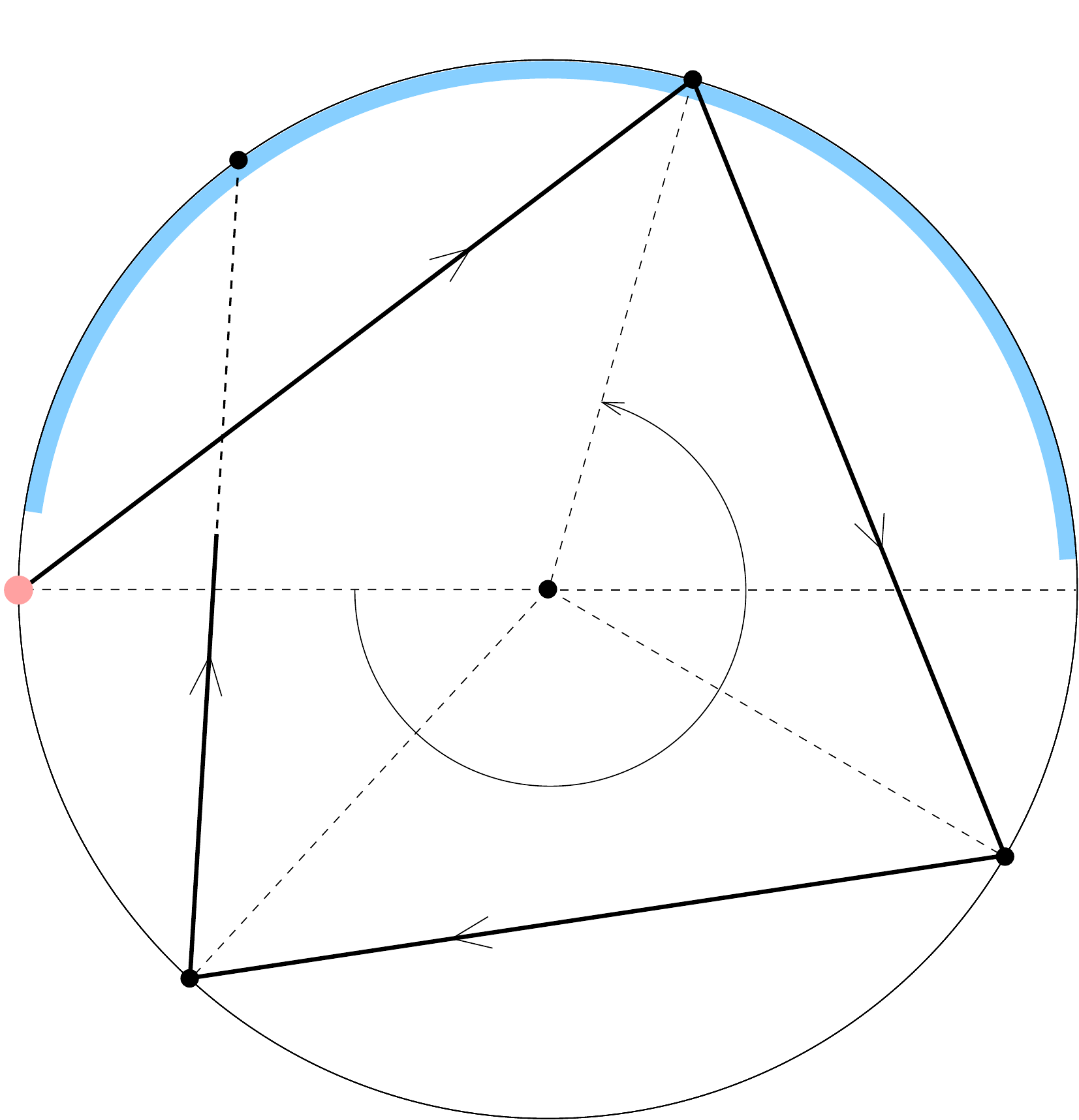} }\qquad 
\subfigure[$t=2 \sin(\alpha/2)$. \label{figdiskrays3b}]
{ \resizebox{6cm}{!}{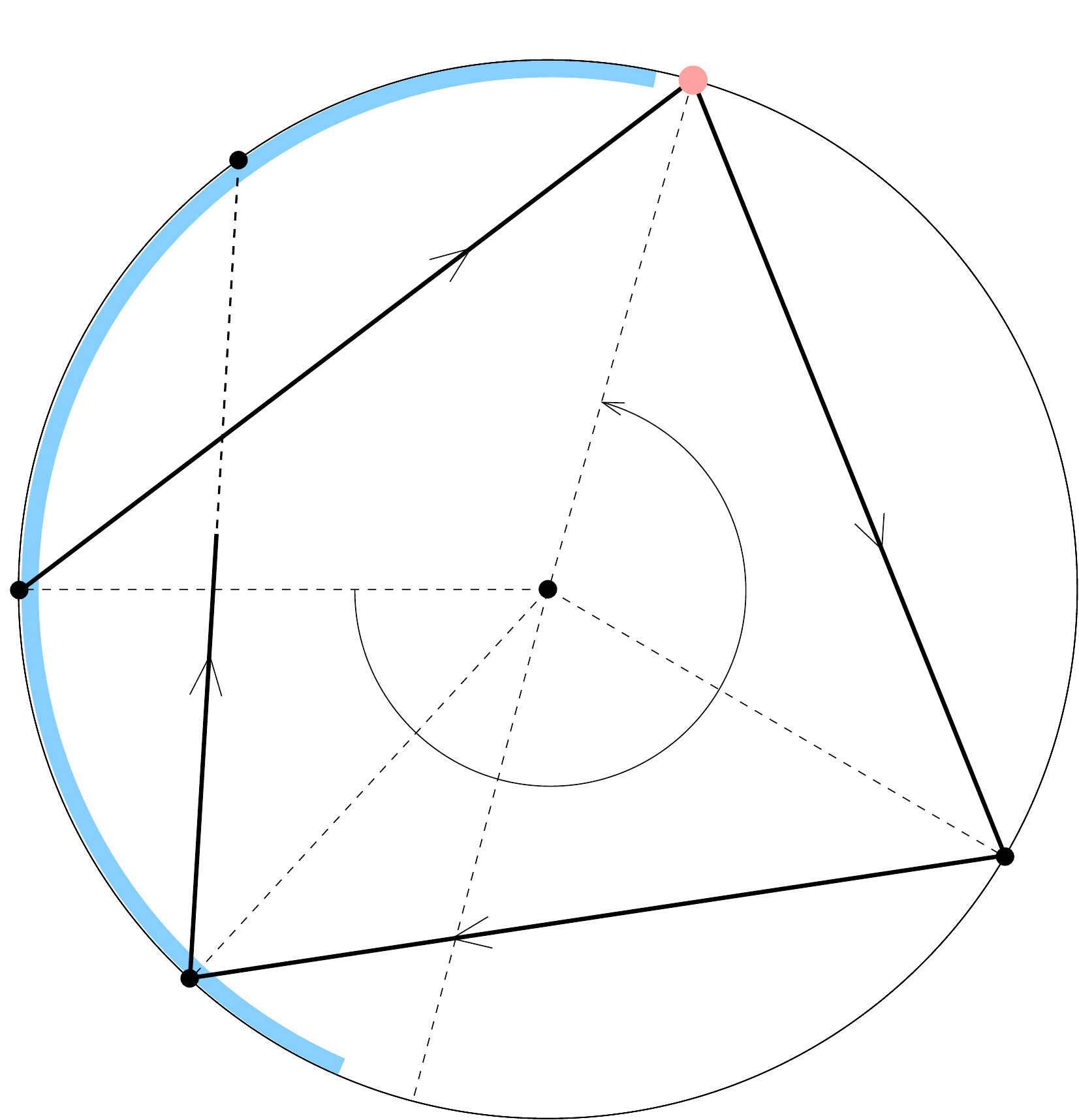} }\quad
\subfigure[$t=4 \sin(\alpha/2)$. \label{figdiskrays3c}]
{ \resizebox{6cm}{!}{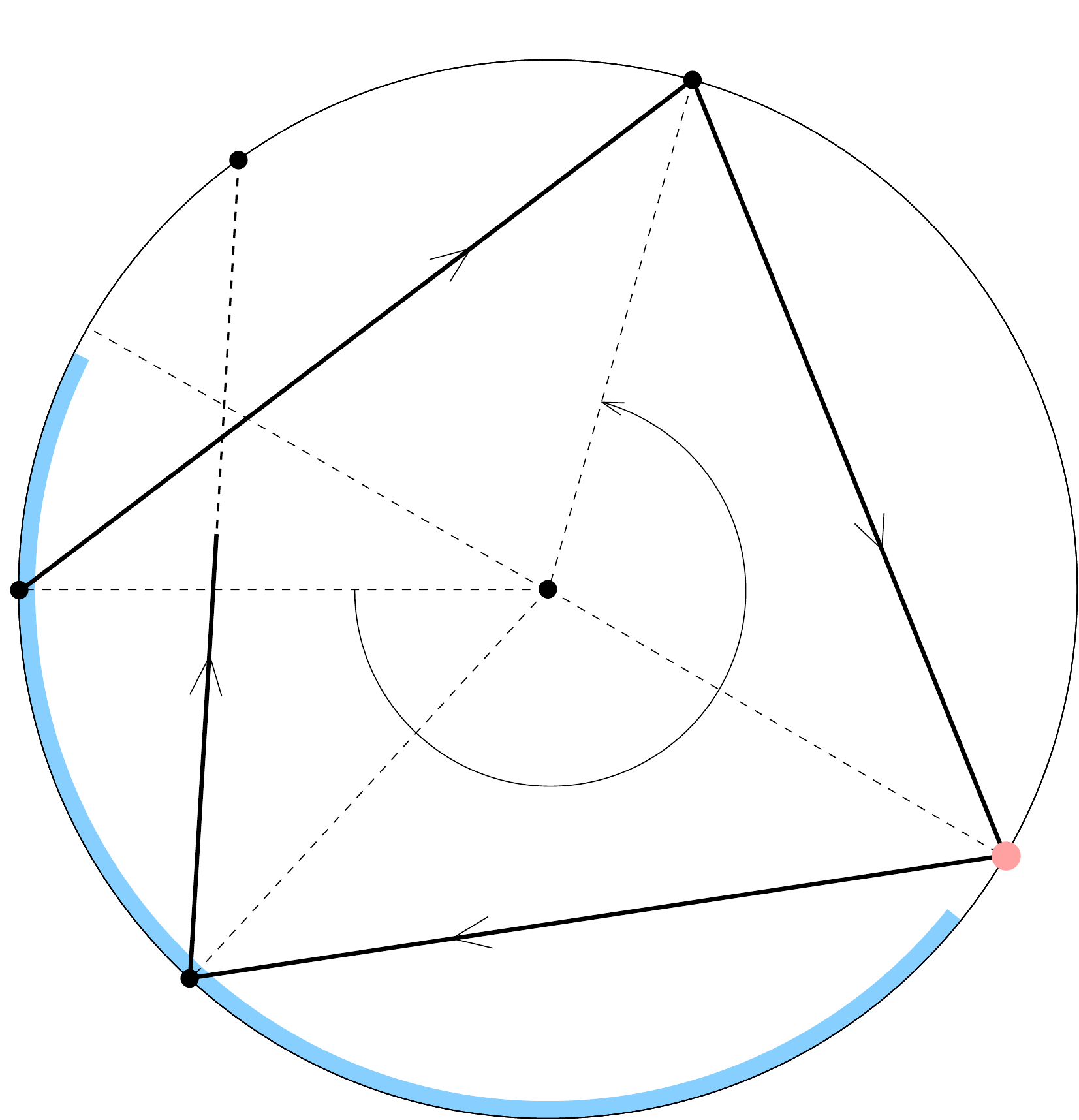} }
\end{center}
\caption{Illustration of property 4 of Proposition~\ref{prop_disk}.}\label{figdiskrays3}
\end{figure}

\item For the fourth property of the previous proposition we obtain a
  moving open set $\omega(t)$ with an ``angular measure'' that is less
  than $\pi$, that is $0< a < \pi$ (see Figure~\ref{figdiskrays3}). In fact, if one allows for
  $\omega(t)$ to be not connected, but rather the union of two
  connected components, for any velocity $v \in (0,+\infty)$ we can exhibit moving sets whose ``angular
  measure'' is a close as one wants to $2 \pi$ and yet the $t$-GCC does not
  hold. This is illustrated  in Figure~\ref{figdiskrays3bis}.

\begin{figure}[h]
\begin{center}
\subfigure[$t=0$. \label{figdiskrays3abis}]
{ \resizebox{4.2cm}{!}{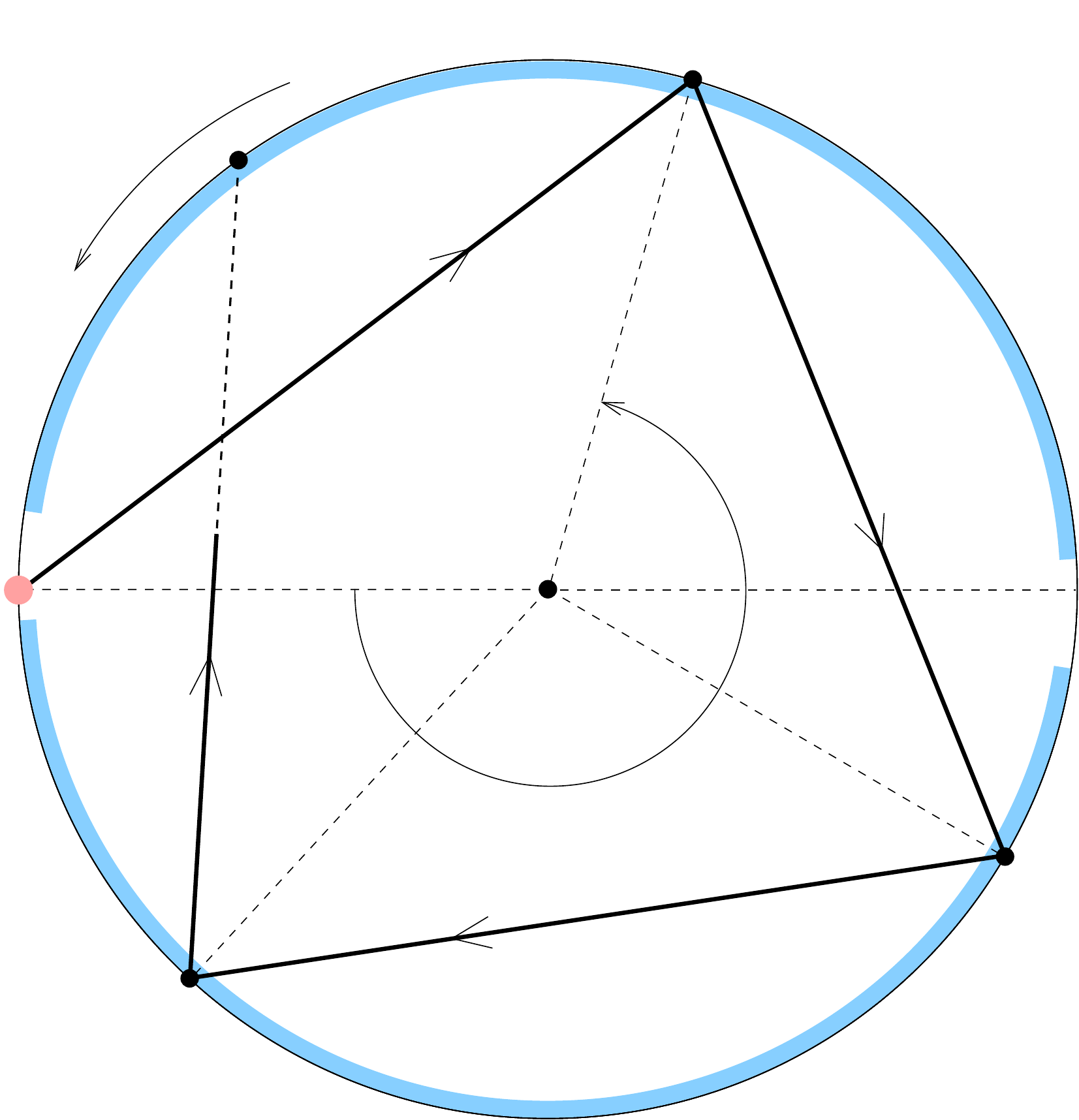} }\quad 
\subfigure[$t=2 \sin(\alpha/2)$. \label{figdiskrays3bbis}]
{ \resizebox{4.2cm}{!}{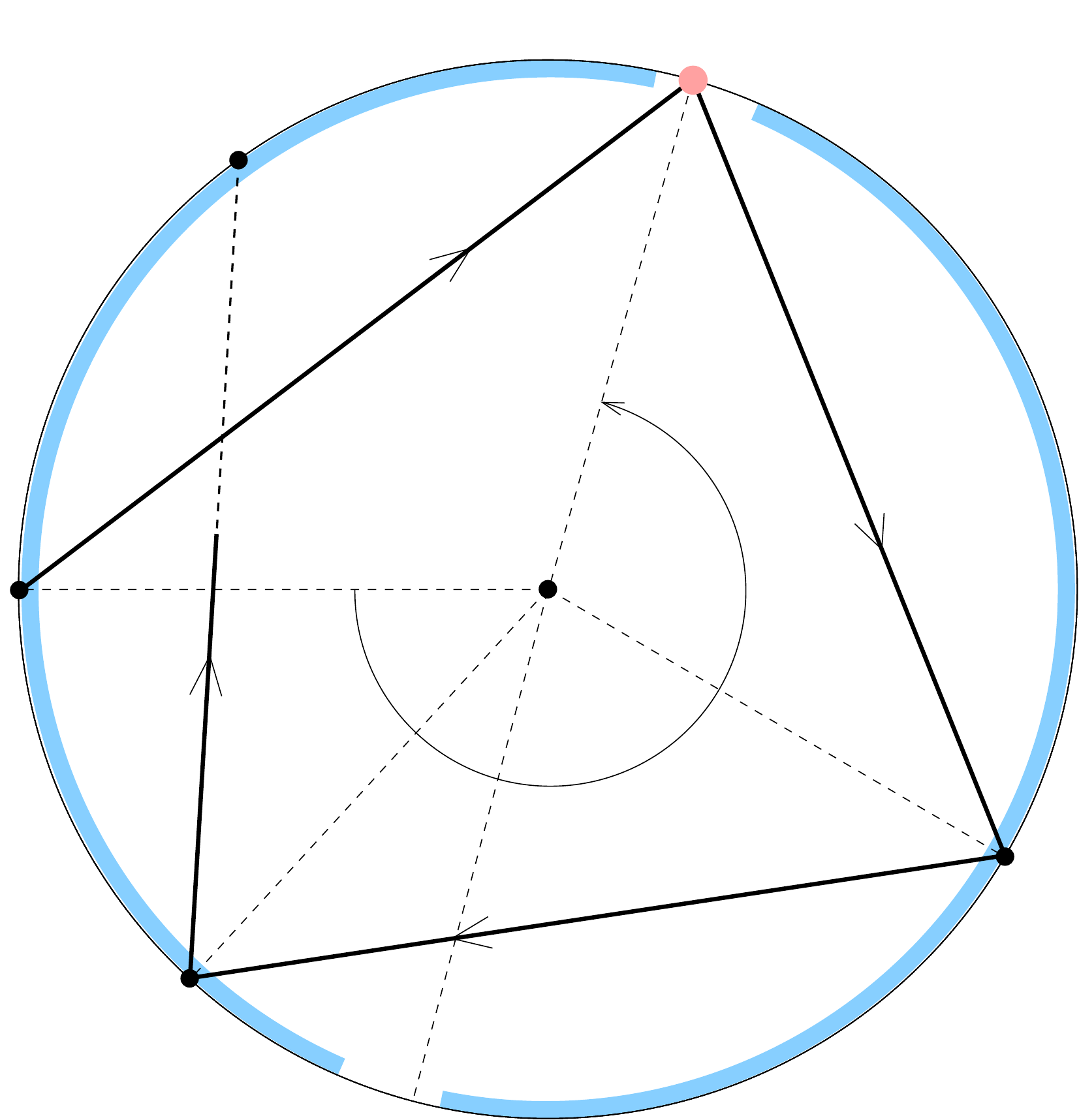} }\quad
\subfigure[$t=4 \sin(\alpha/2)$. \label{figdiskrays3cbis}]
{ \resizebox{4.2cm}{!}{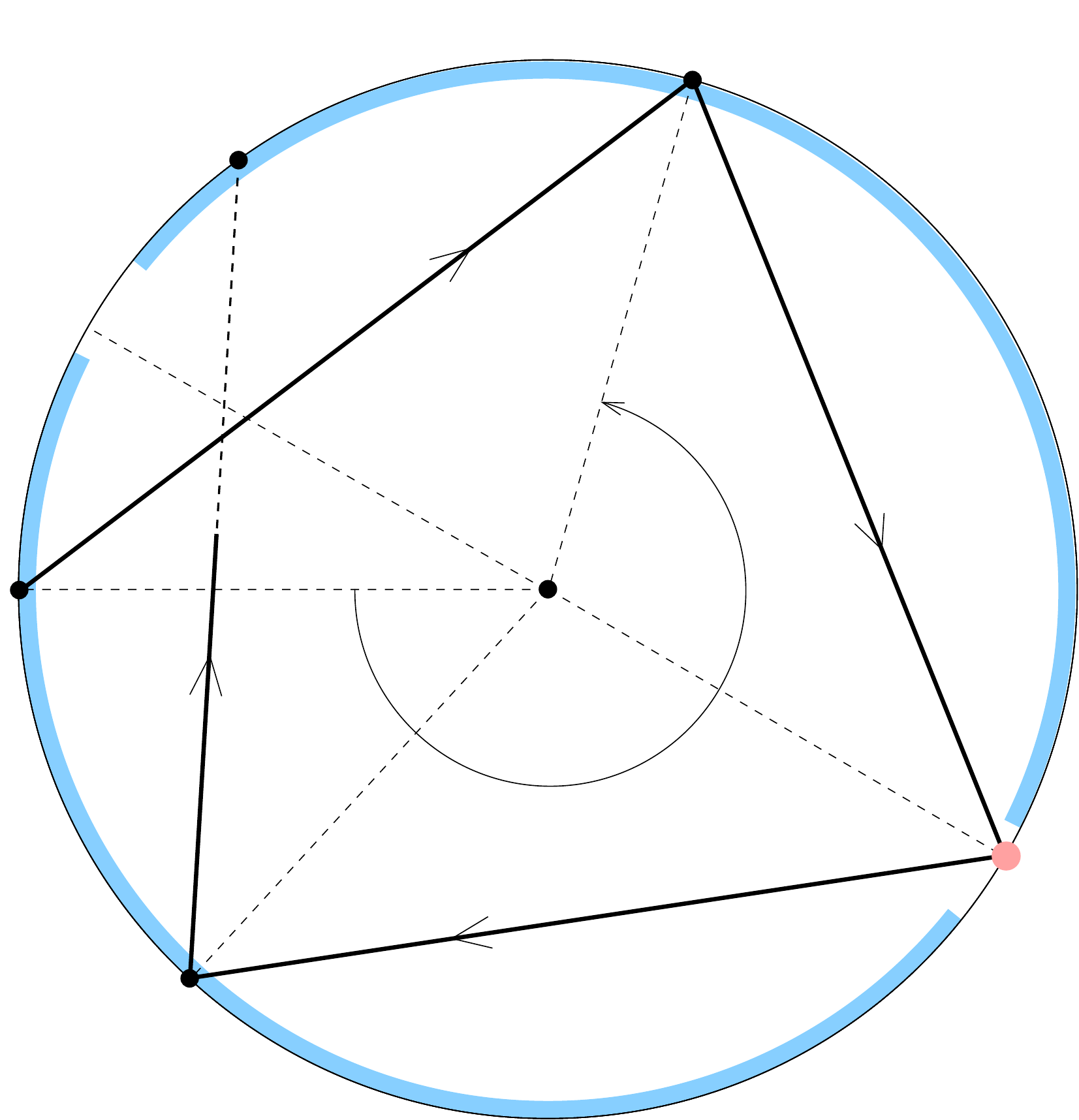} }
\end{center}
\caption{Case of a non connected moving open set $\omega(t)$ moving
  anticlockwise, non statisfying the $t$-GCC with yet a very large
  ``angular'' measure.}\label{figdiskrays3bis}
\end{figure}

\begin{figure}[h]
\begin{center}
\subfigure[$n=2$, $\pi< \alpha < \pi + \delta$. \label{figdiskrays4a}]
{ \resizebox{4.5cm}{!}{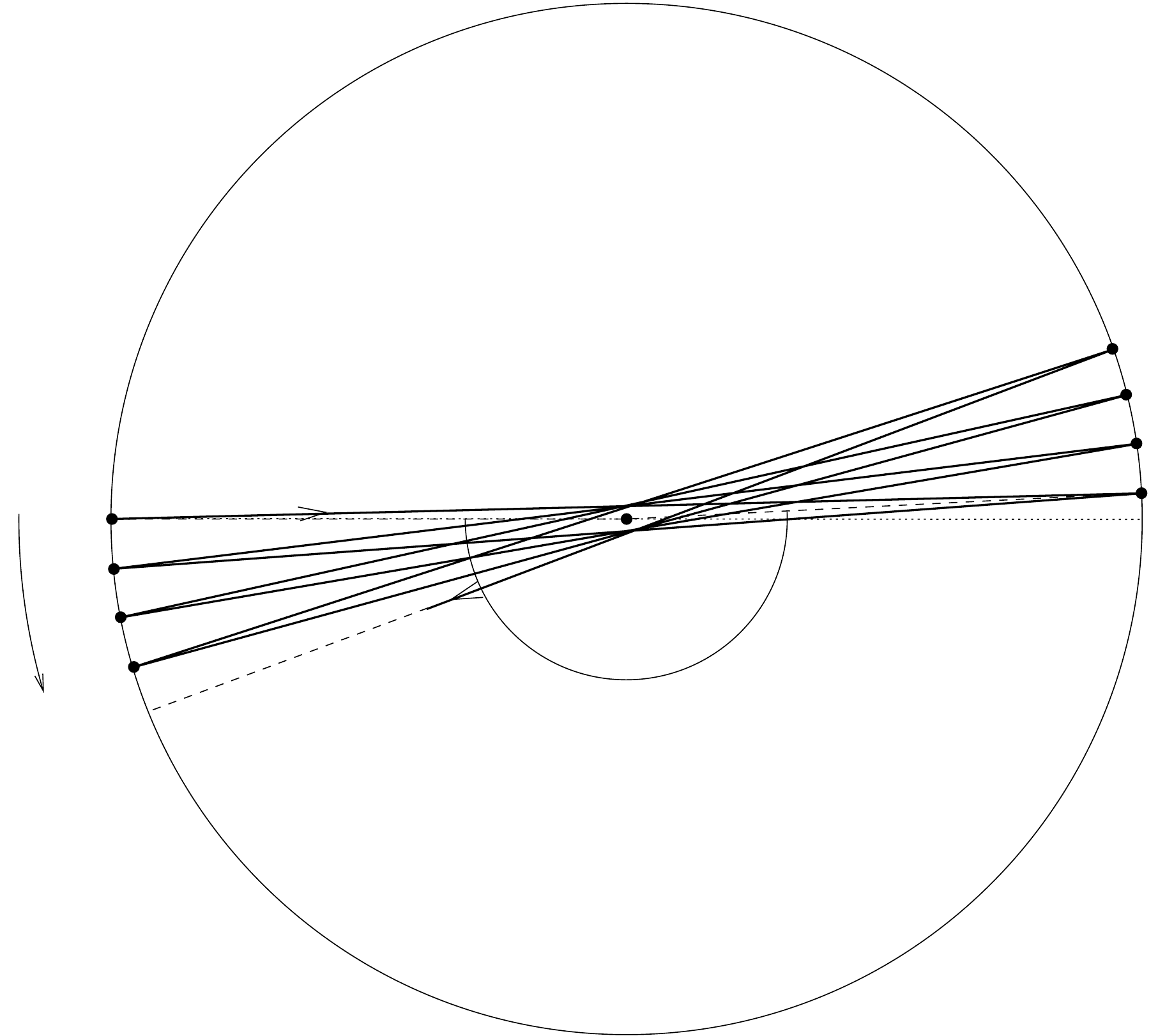} }\quad 
\subfigure[$n=3$, $4\pi/3< \alpha < 4\pi/3 + \delta$. \label{figdiskrays4b}]
{ \resizebox{4.5cm}{!}{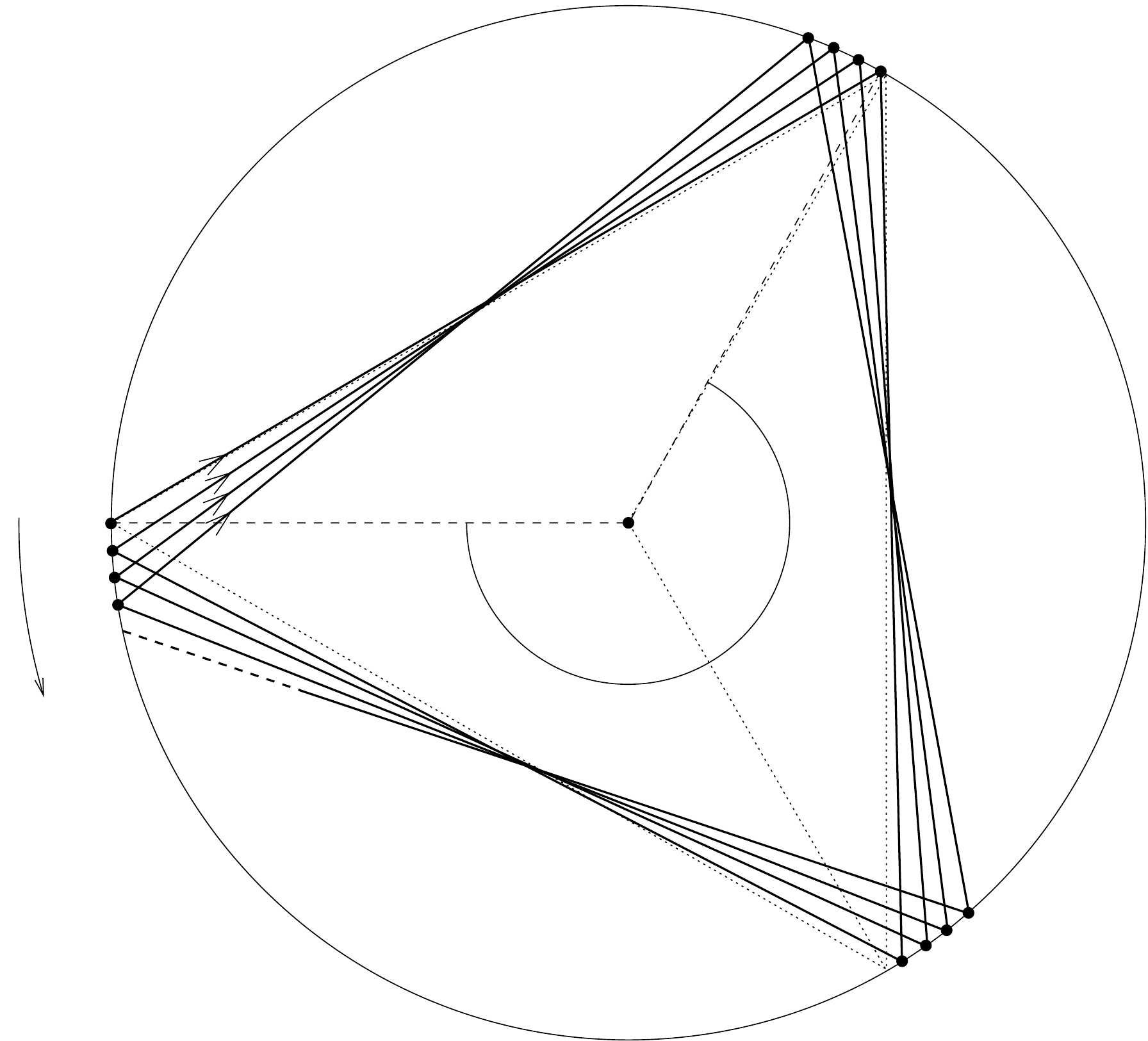} }\quad
\subfigure[$n=4$, $3\pi/2< \alpha < 3\pi/2 + \delta$. \label{figdiskrays4c}]
{ \resizebox{4.5cm}{!}{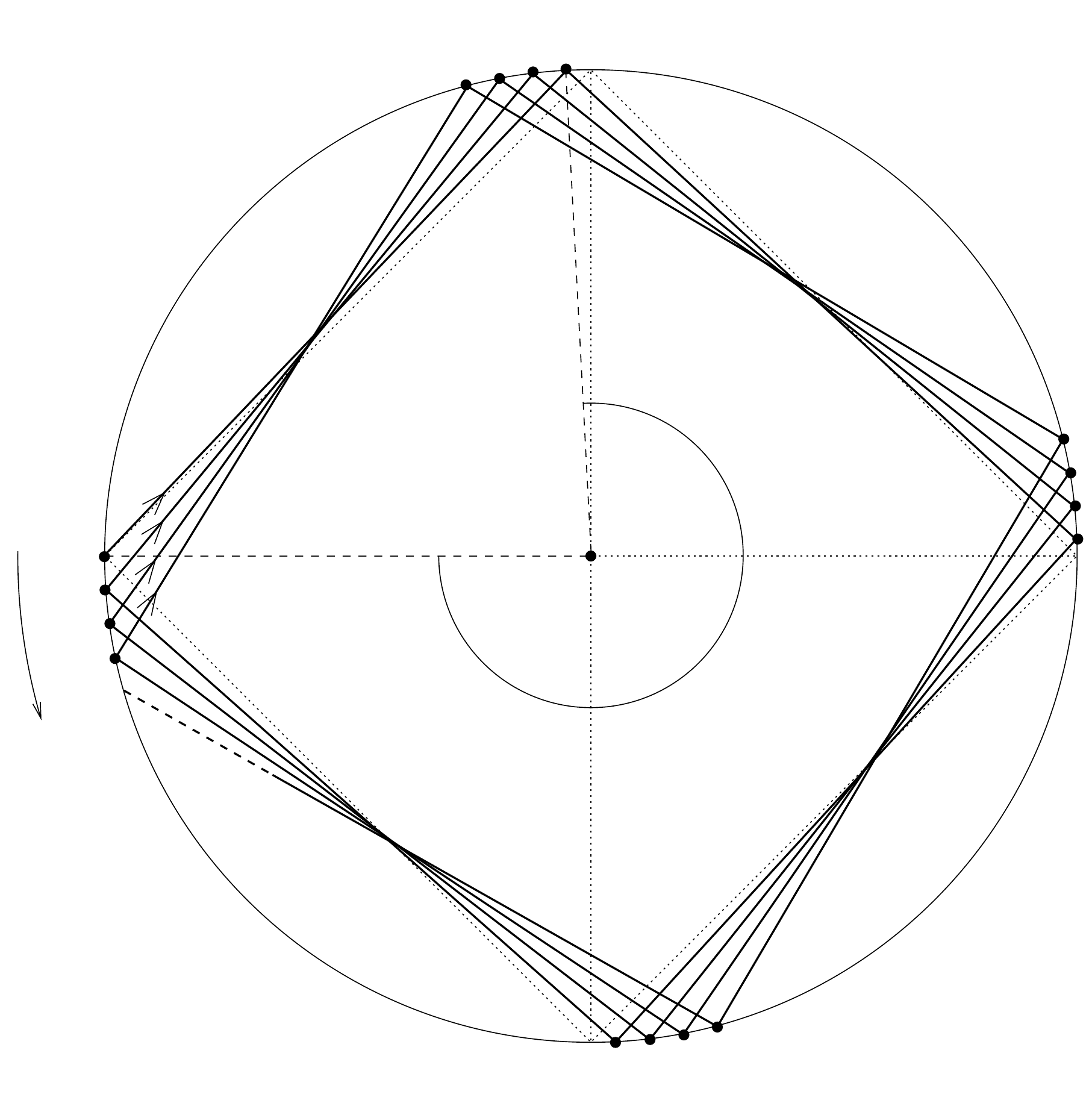} }
\end{center}
\caption{Cases of slow precession speeds with $0 <
  \alpha- 2 \pi\frac{n-1}{n} <\delta$,  with  $\delta$ small and
  $n\geq 2$, that is,
  with a trajectory ``close'' to that of a periodic ray that forms a
  regular polygon (see Figure~\ref{fig: periodic ray}). }\label{figdiskrays4}
\end{figure}

\item 
If $v$ cannot be chosen as large as desired (for physical reasons), 
Proposition \ref{prop_disk} states that the $t$-GCC does not hold true if
$a>0$ and $\eps>0$ are too small. As shown in the proof above, this
lack of observability is due to a secular effect caused by geodesics
whose trace at the boundary produces a pattern that itself varies in
time, with a precession speed that can be tuned to match that of the control
domain. In fact, a precession speed can be obtained as slow as one
wants if $\alpha$ is chosen such that $\pi < \alpha < \pi + \delta$
with $\delta>0$ small. This is illustrated in
Figure~\ref{figdiskrays4a}. 
If one is close to regular polygons, as illustrated in
Figure~\ref{figdiskrays4}, one obtains a precession pattern, that can be used to deduce 
other families of examples of moving domains $\omega(t)$ with
multiple connected components with velocities $v>0$ that do not
satisfy the $t$-GCC.

\item 
Proposition \ref{prop_disk} has been established for the domain drawn in Figure \ref{fig: moving domain on the disk}, sliding anticlockwise along the boundary with a constant angular speed. Other situations can be of interest: we could allow the domain to move with a nonconstant angular speed. For instance, we could allow the domain to move anticlockwise within a certain horizon of time, and then clockwise. This would certainly improve the observability property.
A situation that can be much more interesting in view of practical
issues is to let the angular speed of the control domain evolve
according to $v(t) = v + \beta\sin(\gamma t)$ (with $\beta>0$ small),
that is, with a speed oscillating around a constant value $v$. We
expect that such a configuration, with an appropriate choice of
coefficients, will yield the observability property to be more robust, by
avoiding the situation described in the second property of Proposition
\ref{prop_disk} (non-observability for a dense set of speeds).
\end{enumerate}
\end{remark}

We now state a positive result. 
For $a \in (0,2\pi)$ and $\eps\in (0,1)$,  assume now that the domain $\omega(t)$ is given by $ \omega(t)=\{ (r,\theta)\in [0,1]\times\R\ ;
\  1-\eps< r < 1,\ \theta_0(t)<\theta< a+ \theta_0(t) \}$, with
\begin{align*}
  \theta_0(t) = \begin{cases}
    0 & \text{if}\ 0\leq  t < t_0,\\
    v(t-t_0) & \text{if}\ t_0\leq  t,
  \end{cases}
\end{align*}
for some $t_0>0$ and $v>0$. We set $Q = \cup_{t\geq 0 } \omega(t)$. In this configuration, at first, the domain is still, and then one lets it move. 
\begin{proposition}
\label{eq: positive result disk}
If $4\pi/5 < a < \pi$, $t_0> 2 \pi$, then there exists $0 < v<1$ such that
$T_0(Q, \Omega) < \infty$. 
\end{proposition}

\begin{remark}$\phantom{-}$
\begin{enumerate}
\item
The important aspect of this result lies in the following facts. First, if at
rest, the observability set does not satisfy the geometric control condition;
hence, its motion is crucial for the $t$-GCC to hold. Second, the motion is
performed at a velocity $v$  that is less than that of the wave speed; we
thus have a physically meaningful example. 
\item The result is not optimal as we do not exploit the
  thickness $\eps$ of the domain $\omega(t)$ in the proof. 
\item It would be interesting to further study this ``stop-and-go''
  strategy and see how small the value of $a>0$ can be chosen.
\end{enumerate}
\end{remark}

\begin{proof}[Proof of Proposition~\ref{eq: positive result disk}]
  Let $0< t < t_0$; then $\omega(t) = \omega(0)$ is still.  First, we
  consider the ray associated with $0 \leq \alpha < a$, or
  symmetrically $2 \pi-a< \alpha\leq 2 \pi$, then the movement of the
  successive points $P_k$, $k \in \N$, is anticlockwise, or clockwise,
  respectively; see Figure~\ref{figdiskrays}. Depending on the case
  considered we denote $\beta=\alpha$ or $\beta= 2 \pi- \alpha$.  The
  above condition thus reads better as follows
   $0 \leq \beta < a$.
 In both cases, the (unsigned) angular distance
  between two points is precisely $\beta$. 
  The successive points $P_k$, $k \in \N$,  thus end up meeting $\omega(0)$ in finite time.
  (The case $\beta=0$ coincides with a gliding ray that has angular
  speed 1; it thus meet $\omega(0)$ in finite time.)
 Let us consider $\beta\neq 0$. The maximal number of steps it takes
for any ray associated with $\beta$ to enter $\omega(0)$
is then $\lfloor \frac{2\pi -a}{\beta} \rfloor +1$ yielding a maximal time 
$T_0(\alpha) = 2 \sin(\beta/2) \big( \lfloor \frac{2\pi -a}{\beta} \rfloor +1\big)$,
as the time lapse between two points $P_k$ is $2 \sin(\beta/2)$. Here, the notation $\lfloor . \rfloor$ stands for the usual floor function.
We thus need $T_0 > \max\big(2 \pi-a, \sup_{0< \beta< a}
T_0(\alpha)\big)$.
The value $2 \pi-a$ accounts for the gliding rays ($\beta=0$). 
 Here, we give a crude upper bound for $T_0(\alpha) $ observing that
\begin{align*}
  T_0(\alpha) \leq \frac{\sin(\beta/2)}{\beta/2} (2 \pi - a+ \beta)
  \leq (2 \pi - a+ \beta) < 2 \pi.
\end{align*}
We thus see that if we choose $t_0 > 2 \pi$ then all the rays
associated with the angle $0 \leq \beta< a$ enter $\omega(0)$ for
$0\leq t < t_0$. 

Second, we consider $t\geq t_0$ and we are left only with the rays
that are associated with $a \leq \alpha\leq 2 \pi -a$.
For these rays we consider the two sequences of points $(P_{2k})_k$
and $(P_{2k+1})_k$. The time lapse for a ray
to go  from one point to the consecutive point in these two sequences
is $4 \sin(\alpha/2)$ and this is associated with the (signed) angle
$2(\alpha-\pi)$.
In fact, as $a>\pi/2$, if $a\leq \alpha\leq \pi$, then both sequences
move clockwise and if $\pi \leq \alpha \leq 2\pi -a$, then both sequences
move anticlockwise.

If $v>0$ is the angular speed of
$\omega(t)$ for $t\geq t_0$ then we require $2 (\alpha - \pi) <  4 \sin(\alpha/2) v <
a + 2 (\alpha - \pi)$, that is 
\begin{align}
  \label{eq: encadrement v}
  &\frac{(\alpha - \pi)}{2 \sin(\alpha/2)} <   v < \frac{a + 2 (\alpha - \pi)}{4 \sin(\alpha/2)},
  \qquad \text{for}\ a \leq  \alpha \leq 2\pi -a.
\end{align} 
Observe that $a + 2 (\alpha - \pi)>0$ as $a > \pi/2$.  The left
inequality in \eqref{eq: encadrement v} is necessary as it implies
that the anticlockwise moving open set $\omega(t)$ will be faster than the
two sequences given above, a necessary condition to be able to catch
points in thoses sequences. In particular, this necessary condition is
clearly filled  if the sequences move
clockwise, that is if $a \leq \alpha \leq \pi$.  The right inequality in  \eqref{eq: encadrement v}
expresses that $\omega(t)$ will not turn too fast and then miss the
discrete sequences of points. In fact, during a time interval of length
$4 \sin(\alpha/2)$ the relative angular displacement of the sequence
and the moving set $\omega(t)$ is $\ell = 4 \sin(\alpha/2) v -2
(\alpha - \pi)$ and with \eqref{eq: encadrement v}  we have $0 < \ell < a$. This expresses that the
sequence points cannot be missed.

  As both bounds of \eqref{eq: encadrement
  v} are increasing functions for $\alpha \in (a,2\pi-a)$ we obtain
the following sufficient condition
\begin{align*}
  %\label{eq: encadrement v 2}
  &\frac{(\pi -a)}{2 \sin(a/2)} <   v < \frac{3 a - 2 \pi}{4 \sin(a/2)}.
\end{align*} 
We see that it can be satisfied if $a> 4 \pi /5$. Observing that
$a\mapsto h(a)=\frac{(3a - 2\pi)}{4 \sin(a/2)} $ increases on $(0, \pi]$ and $ 
0< h(\pi)= \pi/4< 1$ we see that the found admissible velocities are such that $0< v< 1$.
\end{proof}

\subsection{A moving domain in a  square}\label{sec_square}
Let $M=\R^2$ (Euclidean) and $\Omega=(0,1)^2$, be the unit square.  We
recall that, as discussed in Remark \ref{rem-bord}--\eqref{rem-bord2},
the statement of Theorem \ref{mainthm} is still valid in the square,
because the generalized bicharacteristic flow is well defined.

Let $a\in(0,1)$. We consider the fixed domain $\tilde{\omega}_0=(-a,a)^2$,
a square  centered at the origin $(0,0)$, and we set 
$\omega_0 = \tilde{\omega}_0 \cap \Omega$
(see Figure \ref{figsquare}).  Since there are periodic rays, bouncing
back and forth between opposite sides of the square, that
remain away from $\omega_0$, the GCC
does not holds true for $\omega_0$, and the wave equation cannot be
observed from the  domain $\omega_0$ in the sense of~\eqref{observability}.

Now, for a given $T>0$, consider a continuous path
$t\in[0,T]\mapsto (x(t),y(t))$ in the closed square
$[0,1]^2$, with $(x(0),y(0))=(0,0)$. We set
$$
\tilde{\omega}(t) = (x(t)-a,x(t)+a)\times (y(t)-a,y(t)+a) \quad
\text{and} \quad \omega(t) = \tilde{\omega}(t) \cap \Omega.
$$
To avoid the occurrence of periodic rays, as described above, that
never meet $\omega(t)$ a necessary condition for the $t$-GCC to hold true
is 
$$
[a,1-a] \subset x([0,T])\quad\textrm{and}\quad [a,1-a] \subset y([0,T]).
$$

Let us assume that the point $(x(t),y(t))$ moves precisely along the boundary of
the square $[0,1]^2$, 
anticlockwise, and with a constant speed $v$. The path $t \mapsto
(x(t), y(t))$ only exhibits singularities when reaching a corner of
the square $[0,1]^2$, where the direction of the speed is
discontinuous (see Figure \ref{figsquare}).

\begin{figure}[h]
\begin{center}
\resizebox{6cm}{!}{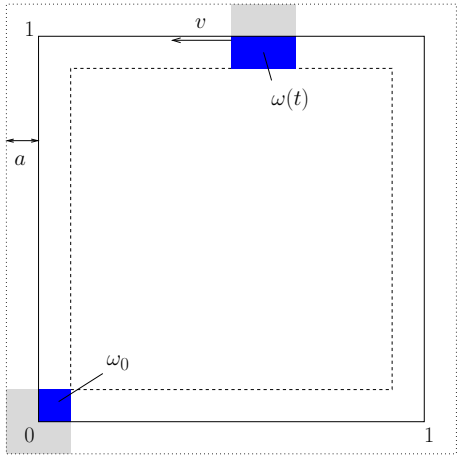}
\end{center}
\caption{Time-varying domain in the square $(0,1)^2$.}\label{figsquare}
\end{figure}
We denote by $T_0(v,a)$ the control time.

\begin{proposition}\label{prop_square}
  We have the following two results:
\begin{enumerate}
\item Let $a\in(0,1/2)$ be arbitrary. For $v > v_0 = (2 -a)/a$ we have 
  $T_0(v,a)<+\infty$,  and moreover 
  $T_0(v,a)\sim \max(\sqrt{2}(1-2a),0)$ as $v\rightarrow +\infty$.

\item If $v\in \cup_{(p,q)\in \N} \sqrt{p^2+q^2}\Q$, there exists $a_0>0$ such
that $T_0(v,a)=+\infty$ if $a\in(0,a_0)$. 
\end{enumerate}
\end{proposition}

\begin{proof}[Proof of Proposition \ref{prop_square}.]
  The argument for the first property is the same as that developed
  in the other examples. For $v>0$ large, the situation becomes
  intuitively as if we have a static control domain that forms a
  $a$-thick strip along the boundary of the square. For this case,
  the geometric control time is $\sqrt{2} (1-2a)$ if $a \in (0,1/2)$
  and $0$ otherwise. 
  More precisely, if a ray enters this strip, it remains in it at
  least for a time $2a$. During such time, it travels at most a
  lateral distance equal to $2a$ (wave speed is one). If during that time the control
  domain goes all around and travels also the additional $2a$
  distance, we can be sure that this ray will be ``caught'' by the
  moving domain. For $v> v_0 =  (4 +2a)/2a = (2 -a)/a$,
  the control time can thus be estimated by $\max(\sqrt{2}(1-2a), 0) \leq T_0(v,a)
  \leq \max(\sqrt{2}(1-2a), 0)+ (4+2a)/v$. Hence, the announced asymptotics.

For the second property, 
as for the other examples, considering $a$ small at first, and thus
the set $\omega(t)$  to be a simple point running along the boundary, greatly helps intuition. 
We start by considering $2$-periodic rays that bounce back and forth between two
opposite sides of the square. They reflect at boundaries at times $t_k
= t_0 + k$, $k\in \Z$, for some $t_0\in\R$. If $v=\frac{p}{q}$ is rational, then the
positions of  the ``moving point'' $\omega(t)$ at times $t_k$ range
over a finite number of points. One can thus easily identify
$2$-periodic rays that never meet that moving point. This property
remains true if  $a>0$ is chosen \suff small.

Let us consider more general periodic rays. All rays propagating in
the square can be described as follows. Let $(x_0,y_0)\in[0,1]^2$ be
arbitrary. Let us consider a ray $t\mapsto (x(t),y(t))$ starting from
$(x_0,y_0)$ at time $t=t_0$, with a slope
$\tan (\alpha) \in \ovl{\R}$, for some $\alpha \in (-\pi,
\pi]$.
Setting $c= \cos(\alpha)$ and $s = \sin(\alpha)$, we define 
$\tilde{x}(t) = x_0+(t-t_0) c$ and $\tilde{y}(t) =y_0+(t-t_0)s$.  Then, for times
$t\in \R$ such that $|t - t_0|$ is small (possibly only if $t\leq t_0$
or $t \geq t_0$), the ray is given by
$t \mapsto (\tilde{x}(t),\tilde{y}(t) )$. 
Introducing $\hat{x}(t), \hat{y}(t) \in [0,2)$ such that $\hat{x}(t)
=\tilde{x}(t)  \mod 2$ and $\hat{y}(t)
=\tilde{y}(t)  \mod 2$
it can be
seen, by ``developing the square" by means of plane
symmetries, that the ray is given by
\begin{equation}
  \label{eq: def x(t) y(t) mod 2}
x(t) = \begin{cases}
\hat{x}(t)& \textrm{if $0\leq \hat{x}(t)\leq 1$}, \\
2- \hat{x}(t) & \textrm{if $1< \hat{x}(t)<2$},
\end{cases}
\qquad 
y(t) = \begin{cases}
\hat{y}(t)& \textrm{if $0\leq \hat{y}(t)\leq 1$}, \\
2- \hat{y}(t) & \textrm{if $1< \hat{y}(t)<2$}.
\end{cases}
\end{equation}
A ray is periodic if and only if
$\tan(\alpha) =p/q\in\Q \cup \{+\infty, -\infty\}$, with $p$ and $q$
relatively prime integers (including the case $q=0$). The period
is equal to $2\sqrt{p^2+q^2}$. In this case, we have
$c=q / \sqrt{p^2+q^2}$ and $s=p /\sqrt{p^2+q^2}$.  
Such a  ray reflects from the boundary of the square at times $t$ in the
union of the following (possibly empty) subsets of $\R$:
\begin{align*}
  A_x = t_0 + \{ (k-x_0)/c\ \mid\ k \in \Z \}, \qquad A_y  = t_0 + \{ (k-y_0)/s\ \mid\ k \in \Z \}.
\end{align*}
The set  $M = M(x_0,y_0, p,q)$ of associated  points where this ray meets the boundary
is then finite and independent of $t_0$.

 If now $v =   r \sqrt{p^2+q^2}$ with $r\in \Q^+$, at times
 $t \in A_x \cup A_y$,  the ``moving point'' $\omega(t)$ ranges over a
 finite set $L = L(x_0, y_0, t_0,p,q,r)$ of points on the boundary of the square,
 as the accumulated distance travelled along the boundary of the square $(0,1)^2$ is of the
 form
 $d_k = t_0 v + (k-x_0)  r (p^2 + q^2)/q$ or $d_k' = t_0 v + (k-y_0)
 r (p^2 + q^2)/p$ and simply
 needs to be considered modulo 4. Adjusting the value of the 
 time $t_0$, we can enforce $M \cap  L = \emptyset$. Hence, the
 associated ray never meets the moving point $\omega(t)$. Finally, as
 the number of points is finite, this  property
remains true if  $a>0$ is chosen \suff small.
\end{proof}

% Similarly, consider now $2\sqrt{2}$-periodic rays that travel back and
% forth along one of the two diagonals. They reflect at boundaries at times $t_k
% = t_0 + \sqrt{2} k$, $k\in \Z$, for some $t_0\in \R$. If $v =
% \sqrt{2}\frac{p}{q}$ with $p$ and $q$ integers, then the
% positions of  the ``moving point'' $\omega(t)$ at times $t_k$ range
% over a finite number of points. Adjusting the time $t_0$ then allows
% one to make sure that none of these points is a corner of the square. 
% We have then found a $2\sqrt{2}$-periodic rays that does not meet this
% moving point and this  property
% remains true if  $a>0$ is chosen \suff small.

\begin{remark}
  If $\tan(\alpha)\in\R\setminus\Q$, then the
  set of points at which the corresponding ray reflects at the boundary
  $\d\Omega$ is dense in $\d\Omega$. In fact at such a point, using
  the parametrization given in the proof above,   we
  have either $\tilde{x}(t) = x_0 + (t-t_0) \cos(\alpha)\in\Z$, or
  $\tilde{y}(t) = y_0 + (t-t_0) \sin(\alpha)\in\Z$. For instance,
  if $\tilde{x}(t)=2k\in\Z$, that is, $x(t) =0$, meaning that we consider point on the
  left-hand-side vertical side of the square, then the corresponding $\tilde{y}(t)$ satisfies
  $\tilde{y}(t)  = (y_0 - x_0 \tan(\alpha)) + 2 k \tan(\alpha)$. 
  Using the fact that the set  that $\{ 2k \beta \mod 2\ \ \mid\ k\in\Z\}$ is dense in $[0,2]$ if and only if
  $\beta\in\R\setminus\Q$, we conclude that $\hat{y}(t) = \tilde{y}(t) \mod 2$ is
  dense in $[0,2]$ and thus $y(t)$ is dense in $[0,1]$ considering~\eqref{eq: def x(t) y(t) mod 2}.
  From this density, we deduce that the analysis in the case
  $\tan(\alpha)\in\R\setminus\Q$ may be quite intricate.
\end{remark}

\subsection{An open question}
Let $\Omega$ be a domain of $\R^d$. Let $\omega_0$ be a small disk in
$\Omega$, of center $x_0\in\Omega$ and of radius $\eps>0$.  Let $v>0$
arbitrary.  Given a path $t\mapsto x(t)$ in $\Omega$, we define
$\omega(t)=B(x(t),\eps)$ (an open ball), with $x(0)=x_0$.  We say that
the path $x(\cdot)$ is admissible if $\omega(t)\subset\Omega$ for
every time $t$.  We raise the following question:
\begin{align}
\tag{$\star$}
\begin{array}{l}
\text{Do there exist $T>0$ and an admissible $\Con^1$ path $t\mapsto
  x(t)$ in $\Omega$, with speed less than or}\\ 
\text{equal to $v$, such that $(Q,T)$ satisfies the $t$-GCC?}
\end{array}
\label{open question}
\end{align}
Here, we have set
$Q = \{ (t,x)\in\R\times\Omega\ \mid\ t\in\R,\ x\in\omega(t)\}$.  Of
course, many variants are possible: the observation set is not
necessarily a ball, its velocity may be constant or not. The examples
of Sections~\ref{sec: ex dim 1}--\ref{sec_square} have shown that addressing the question
\eqref{open question} is far from obvious.  Assumptions on the domain
$\Omega$ could be made; for instance, one may wonder whether
ergodicity of $\Omega$ may help or not.

Note that, above in \eqref{open question}, we restrict the speed of $t\mapsto x(t)$. In fact, using arguments as in the beginning of both the proofs of Proposition \ref{prop_disk} and Proposition \ref{prop_square}, if $t\mapsto x(t)$ is periodic and if $\cup_{t\in\R} B(x(t),\eps)$ satisfies the ``static" Geometric Control Condition, then for a sufficiently large speed of the moving point, the set $Q$ satisfies the $t$-GCC. An estimate of the minimal speed can be derived from the inner diameter of $\cup_{t\in\R} B(x(t),\eps)$.

\section{Boundary observability and control}
\label{sec: boundary control}
In this section, we briefly extend our main results to the case of boundary observability. We consider the framework of Section \ref{sec_1.2}, and we assume that $\d\Omega$ is not empty. We still consider the wave equation \eqref{waveEq}, and we restrict ourselves, for simplicity, to Dirichlet conditions along the boundary.

Let $R$ be an open subset of $\R\times\d\Omega$. We set
$$
\Gamma(t) = \{ x\in\d\Omega\ \mid\  (t,x)\in R\} .
$$
We say that the Dirichlet wave equation is observable on $R$ in time $T$ if there exists $C>0$ such that
\begin{equation}\label{obsfrontiere}
C \Vert (u(0,\cdot),u_t(0,\cdot))\Vert_{H_0^1(\Omega)\times L^2(\Omega)}^2
\leq \int_0^T\!\!\int_{\Gamma(t)} \left|\frac{\d u}{\d n}(t,x)\right|^2 \,d\mathcal{H}^{n-1}\, dt,
\end{equation}
for all solutions of \eqref{waveEq} with Dirichlet boundary conditions. Here, 
$\mathcal{H}^{n-1}$ is the $(n-1)$-Hausdorff measure.

In the static case (that is, if $\Gamma(t)\equiv\Gamma$ does not
depend on $t$), the observability property holds true as soon as
$(\Gamma,T)$ satisfies the following GCC (see \cite{BardosLebeauRauch,BurqGerard}):
\begin{quote}
{\it
  Let $T>0$. The open set $\Gamma \subset \d\Omega$ satisfies the
  Geometric Control Condition if the projection onto $M$ of every
  (compressed) generalized bicharacteristic meets $\Gamma$ at a time
  $t \in (0,T)$ at a nondiffractive point.
}
\end{quote}
Recall the definition of nondiffractive points given in Section~\ref{sec: Gk}. 
The $t$-GCC is then defined similarly to Definition \ref{def_tGCC}.

\begin{definition}%\label{def_tGCC- bord}
  Let $R$ be an open subset of $\R \times \d \Omega$, and let $T>0$.
  We say that $(R,T)$ satisfies the \emph{time-dependent Geometric
    Control Condition} (in short, $t$-GCC) if every compressed
  generalized bicharacteristic
  $\comp \gamma: \R \ni s \mapsto (t(s), x(s), \tau(s), \xi(s)) \in
  \comp T^\ast Y \setminus E$
  is such that there exists $s \in \R$ such that $t(s)\in (0,T)$ and
  $(t(s), x(s)) \in R$ and
  $\comp \gamma(s) \in H \cup G \setminus G_d$ (a hyperbolic point
  or a glancing yet nondiffractive point).
 
  We say that $R$
  satisfies the $t$-GCC if there exists $T>0$ such that $(R,T)$ satisfies
  the $t$-GCC.

  The control time $T_0(R,\Omega)$ is defined by
  \begin{equation*}
    T_0(R,\Omega) = \inf \{ T>0\ \mid\ (R,T)\ \text{satisfies the $t$-GCC} \} ,
  \end{equation*}
  with the agreement that $T_0(R,\Omega)=+\infty$ if $R$ does not
  satisfy the $t$-GCC.
\end{definition}

\begin{theorem}%\label{mainthm-bord}
Let $R$ be an open subset of $\R \times \d \Omega$ that satisfies the $t$-GCC. Let $T>T_0(R,\Omega)$.
We assume moreover that no
(compressed) generalized bicharacteristic has a contact of infinite order with
$(0,T)\times\d\Omega$, that is, $G^\infty  = \emptyset$.
Then, the observability inequality \eqref{obsfrontiere} holds.
\end{theorem}

\begin{proof}
The proof is similar to the proof of Theorem \ref{mainthm} done in
Section \ref{subsec_proofmainthm}. We just point out that the set of invisible solutions is defined by
$$
N_T = \big\{ u \in H^1((0,T)\times\Omega) \ \mid\ u \ \text{is a Dirichlet
  solution of \eqref{waveEq}, and}\ \chi_R\,\frac{\d u}{\d n}=0 \big\} .
$$
If $u \in N_T$ and $\rho \in T^\ast((0,T)\times \Omega)$, we wish to
prove that $u$ is smooth at $\rho$. Let
$\comp \gamma(s) = (x(s), t(s), \xi(s), \tau(s))$ be the compressed
bicharacteristic that originates from $\rho$ (at $s=0$).  There exists
$s_0 \in \R$ such that $t(s_0) \in (0,T)$ and
$(t_0,x_0) = (t(s_0), x(s_0)) \in R$.  To fix ideas, let us assume
that $s_0 \geq 0$ and let us set
$\comp \gamma_0 = \comp \gamma(s) \in T^\ast \d Y$.  Because of
the $t$-GCC, we may assume that $\comp \gamma_0$ is a nondiffractive
point. Note that the case $s_0 \leq 0$ can be treated similarly.

Let now $V$ be an open \nhd of  $(t_0,x_0)$ in $\R \times M$ such
that $V_\d = V \cap  (\R \times \d\Omega) \subset R$. 
In $V$, we extend the function $u(t,x)$ by zero outside $\R\times \Omega$ and
denote this extension by $\udl{u}$. 
Since $u_{|V_\d} = \d_n u_{|V_\d} =0$ we observe that $\udl{u}$ solves
$P \udl{u}=0$ in $V$. 
As $\gamma_0$ is nondiffractive, the natural
bicharacteristic associated with $\gamma_0$ has points outside
$\Omega$ in any \nhd of $\gamma_0$. By propagation of singularities for $\udl{u}$ we thus find that
$\udl{u}$ is smooth at $\gamma_0$. Then, by propagation of singularities
along the compressed generalized bicharacteristic flow (see 
\cite{MelroseSjostrand1,MelroseSjostrand2,Hoermander:V3}), we find that
$u$ is smooth at $\rho$. 
Having $u$ smooth in $(0,T) \times \Omega$, we see that if $u\in N_T$ then $\d_t u\in N_T$. The rest of the proof follows.
\end{proof}

\begin{remark}
By duality, we have, as well, a boundary controllability result, as in
Theorem~\ref{cor1}.
\end{remark}

\begin{remark}
It is interesting to analyze, in this context of boundary observability, the examples of the disk and of the square studied in Section \ref{sec_examples}.
\begin{itemize}
\item For the disk (see Section \ref{sec_disk}): we set
$$
\Gamma(t) = \{ (r,\theta)\in [0,1]\times\R\ \ \mid\  r=1,\ vt<\theta<vt+a \}.
$$
This is the limit case of the case of Section \ref{sec_disk}, with $\eps=0$.
With respect to Proposition \ref{prop_disk}, it is not true anymore
that $T_0(v,a)<+\infty$ if $v$ is chosen \suff large. 

The other items of Proposition \ref{prop_disk}, providing sufficient conditions such that $T(v,a)=+\infty$, are still valid.

\item For the square (see Section \ref{sec_square}): $\omega(t)$ is the translation of the segment $(0,2a)$ along the boundary, moving anticlockwise and with constant speed $v$.

With respect to Proposition \ref{prop_square}, it is not true anymore that $T_0(v,a)<+\infty$ if $v$ is large enough. The second item of Proposition \ref{prop_square}, providing a sufficient condition such that $T(v,a)=+\infty$, is still valid.
\end{itemize}
We stress that, in Propositions \ref{prop_disk} and \ref{prop_square}, the fact that $T_0<+\infty$ for $v$ large enough was due to the fact that the width of the observation domain is positive.
This remark shows that observability is even more difficult to realize with moving observation domains located at the boundary.
\end{remark}

\appendix
\section{A class of test operators near the boundary}\label{appendice}
We denote by $\Psi^m(Y)$ the space of operators of the form $R = R_{\text{int}} +
R_{\text{tan}}$ where:
\begin{itemize}
  \item 
    $R_{\text{int}}$ is a classical pseudodifferential operator of order $m$
    with compact support in $\R \times \Omega$, that is, satisfying
    $R_{\text{int}} = \varphi R_{\text{int}} \varphi$ for some
    $\varphi \in \Cinfc(\R \times \Omega)$. 
  \item $R_{\text{tan}}$ is a classical tangential pseudodifferential operator
    of order $m$. In the local normal
    geodesic coordinates introduced in Section~\ref{sec: generalized
      bichar} such an operator only acts in the $y'$ variables. 
\end{itemize}
If $\sigma(R_{\text{int}})$ and $\sigma(R_{\text{tan}})$ denote the
homogeneous principal symbols of $R_{\text{int}}$ and $R_{\text{tan}}$
respectively, we observe that their restriction to $\Char_Y(p) \cup
T^\ast (\R \times \d\Omega)$ is
well defined. Then, by means of the map $j: T^\ast Y \to \comp T^\ast
Y$, the function $\sigma(R_{\text{int}})_{|\Char_Y(P)} +
\sigma(R_{\text{tan}}) _{|\Char_Y(P) \cup
T^\ast (\R \times \d\Omega)}$ yields a {\em continuous} map
on $\hat \Sigma = j (\Char_Y(p)) \cup E$, that we denote by $\kappa(R)$. 
Its homogeneity yields a continuous function on $S^\ast \hat \Sigma =
\hat \Sigma / (0,+\infty)$. 

We then have the following proposition (see \cite{Lebeau:96,BurqLebeau}).
\begin{proposition}
  \label{eq: existence measure}
  Let $(u_n)_{n \in \N}$ be a bounded sequence of $H^1(\R\times \Omega)$ that
  satisfies $(\d_t^2 - \triangle_g)u_n =0$ and weakly converges to $0$. 
  Then, there exist a subsequence $(u_{\varphi(n)})_{n \in \N}$ and  a positive measure $\mu$ on $S^\ast \hat \Sigma$
  such that 
  $$
  (R u_{\varphi(n)}, u_{\varphi(n)})_{H^1(\Omega)}\xrightarrow[n\rightarrow+\infty]{}
  \langle \mu, \kappa(R)\rangle, \qquad R \in \Psi^0(Y).
  $$
\end{proposition}
This extends results in the interior introduced in the
seminal works \cite{Gerard,Tartar}.

\end{document}